\documentclass[11pt]{article}

\usepackage{amsmath,amssymb,amsfonts,amsthm,enumitem,xy,pstricks,pst-node,comment,mathrsfs}
\xyoption{all}

\usepackage{footnote}
\usepackage[symbol]{footmisc}

\numberwithin{equation}{section}
\newtheorem{thm}{Theorem}[section] 
\newtheorem{prp}[thm]{Proposition}
\newtheorem{lmm}[thm]{Lemma}   
\newtheorem{crl}[thm]{Corollary} 

\theoremstyle{definition}
\newtheorem{dfn}[thm]{Definition}
\newtheorem{rmk}[thm]{Remark}
\newtheorem{eg}[thm]{Example}

\def\BE#1{\begin{equation}\label{#1}}
\def\EE{\end{equation}}
\def\e_ref#1{(\ref{#1})}

\def\ov#1{\overline{#1}}
\def\ti#1{\widetilde{#1}}
\def\wt#1{\widetilde{#1}}
\def\wh#1{\widehat{#1}}
\def\sf#1{\textsf{#1}}
\def\smsize#1{\begin{small}#1\end{small}}

\def\lra{\longrightarrow}

\def\xlra#1{\xrightarrow{#1}}

\def\eset{\emptyset}
\def\i{\infty}

\def\al{\alpha}

\def\de{\delta}

\def\io{\iota}

\def\si{\sigma}
\def\vp{\varphi}

\def\cA{\mathcal A}
\def\cB{\mathcal B}

\def\D{\mathcal D}

\def\cH{\mathcal H}
\def\sH{\mathscr H}

\def\cI{\mathcal I}
\def\R{\mathbb R}
\def\cS{\mathcal S}
\def\Z{\mathbb Z}

\def\De{\Delta}

\def\Bd{\textnormal{Bd}~\!}

\def\nd{\textnormal{d}}
\def\CH{\textnormal{CH}}

\def\Hom{\textnormal{Hom}}
\def\id{\textnormal{id}}
\def\Id{\textnormal{Id}}
\def\Im{\textnormal{Im}~\!}
\def\Int{\textnormal{Int}~\!}

\def\PD{\textnormal{PD}}

\def\rk{\textnormal{rk}~\!}
\def\sd{\textnormal{sd}\,}
\def\St{\textnormal{St}}
\def\sign{\textnormal{sign}~\!}
\def\supp{\textnormal{supp}}
\def\top{\textnormal{top}}

\def\cal{\mathcal}
\def\Bbb{\mathbb}

\def\ale{\aleph}

\def\Hcl{H^{\textnormal{lf}}}
\def\ovHcl{\ov{H}^{\textnormal{lf}}}
\def\cHcl{{\mathcal H}^{\textnormal{cl}}}
\def\Scl{S^{\textnormal{lf}}}
\def\ovScl{\ov{S}^{\textnormal{lf}}}
\def\prScl{S'^{\textnormal{lf}}}

\def\hb{\hbar}
\def\prt{\partial}

\begin{document}

\title{Pseudocycles for Borel-Moore Homology} 
\author{Spencer Cattalani\thanks{Partially supported by NSF grant DMS 1901979}
~and Aleksey Zinger\thanks{Partially supported by NSF grant DMS 1901979}}
\date{\today}
\maketitle

\begin{abstract}
\noindent
Pseudocycles are geometric representatives for integral homology classes
on smooth manifolds that have proved useful in particular for defining 
gauge-theoretic invariants.
The Borel-Moore homology  is often a more natural object
to work with in the case of non-compact manifolds than the usual homology.
We define weaker versions of the standard notions of pseudocycle and
pseudocycle equivalence and then describe a natural isomorphism between 
the set of equivalence classes of these weaker pseudocycles and
the Borel-Moore homology.
We also include a direct proof of a Poincar\'e Duality between the singular cohomology
of an oriented manifold and its Borel-Moore homology.
\end{abstract}

\thispagestyle{empty}
\tableofcontents

\section{Introduction}

\subsection{Main theorem}

\noindent
Constructions of some important gauge-theoretic invariants involve representing
cohomology classes on a smooth manifold~$X$ geometrically.
As the submanifolds in~$X$ and embedded cobordisms between them do not generally suffice
for  representing the singular homology $H_*(X;\Z)$ of~$X$,
\sf{pseudocycles} have been used as a suitable replacement in the case of compact manifolds.
For example, pseudocycles are central to the constructions of Gromov-Witten invariants
for compact semi-positive symplectic manifolds
in \cite[Section~7.1]{MS12} and \cite[Section~1]{RT}.\\

\noindent
The \sf{Borel-Moore} \sf{homology} $\Hcl_*(X;\Z)$ of a topological space~$X$,
also known as the homology with closed supports and the homology based on locally finite chains,
is introduced from a sheaf-theoretic perspective in~\cite{BM60}.
If $X$ is compact, $\Hcl_*(X;\Z)$ is just the usual singular homology $H_*(X;\Z)$.
On the other hand, a closed oriented $k$-submanifold~$M$ in a manifold~$X$ determines
a class $[M]_X$ in $\Hcl_k(X;\Z)$, even if $M$ is not compact.
If $X$ is an oriented $n$-manifold, $\Hcl_k(X;\Z)$ is Poincar\'e dual to 
the singular cohomology $H^{n-k}(X;\Z)$ with respect to their pairing 
with the compactly supported cohomology $H_c^k(X;\Z)$.
The purpose of the present paper is to provide an analogue of pseudocycles for
the Borel-Moore homology of a non-compact manifold~$X$ and a geometric way of representing
all cohomology of~$X$.
As indicated in \cite[Section~1]{FZ1}, \cite[Section~1]{FZ2}, and \cite[Sections~3a,3b,5a.2,5c]{Wilkins},
classes on non-compact manifolds can be relevant even if one is interested
only in compact manifolds.\\

\noindent
A subset~$Z$ of a manifold~$X$  is \sf{of dimension at most~k}, 
which we write~as
$$\dim\,Z\le k\,,$$
if there exists a \hbox{$k$-dimensional} manifold~$Y$ and 
a smooth map $h\!:Y\!\lra\!X$ such that \hbox{$Z\!\subset\!h(Y)$}.
If \hbox{$f\!:M\!\lra\!X$} is a continuous map between topological spaces,
\sf{the boundary of~$f$} is the~subspace
$$\Bd f\equiv \bigcap_{K\subset M\text{~cmpt}}\!\!\!\!\!\!\ov{f(M\!-\!K)}\subset X,$$
where the overline $\ov{~\cdot~}$ denotes the closure in~$X$.
If $A\!\subset\!X$ is a closed compact subset disjoint from~$\Bd f$,
then $f^{-1}(A)\!\subset\!M$ is compact.
A continuous map~$f$ as above is called \sf{proper} 
if $f^{-1}(A)\!\subset\!M$ is compact for every
compact subset $A\!\subset\!X$.
If $\Bd f\!=\!\eset$ and $X$ is Hausdorff, then $f$ is proper.
If $f$ is proper and $X$ is locally compact, 
i.e.~every point of~$X$ has an arbitrarily small precompact open neighborhood,
then $\Bd f\!=\!\eset$.
If $f\!:M\!\lra\!X$ is a continuous map from a compact space to a Hausdorff one,
then $f$ is proper.

\begin{dfn}\label{BMpseudo_dfn} 
Let $X$ be a smooth manifold.
\begin{enumerate}[label=(\alph*),leftmargin=*]	

\item A smooth map $f\!:M\!\lra\!X$ is a \sf{Borel-Moore $k$-pseudocycle} if
$M$ is an oriented $k$-manifold and $\dim\Bd f\!\le\!k\!-\!2$.

\item\label{BMequiv_it} Two Borel-Moore $k$-pseudocycles 
$f_0\!:M_0\!\lra\!X$ and $f_1\!:M_1\!\lra\!X$
are \sf{equivalent} if there exist a smooth oriented manifold~$\wt{M}$,
a smooth map $\wt{f}\!:\wt{M}\!\lra\!X$, and
closed subsets $Y_0\!\subset\!M_0$ and $Y_1\!\subset\!M_1$ such~that
\begin{gather*}
\dim Y_0,\dim Y_1\le k\!-\!1, \quad \prt\wt{M}=(M_1\!-\!Y_1)\!\sqcup\!-(M_0\!-\!Y_0),\\
\dim\Bd\wt{f}\le k\!-\!1,\quad 
\wt{f}|_{M_0-Y_0}=f_0|_{M_0-Y_0}, \quad \wt{f}|_{M_1-Y_1}=f_1|_{M_1-Y_1}.
\end{gather*}

\item The \sf{$k$-th Borel-Moore pseudocycle group} is the set $\cHcl_*(X)$ of 
equivalence classes of Borel-Moore $k$-pseudocycles to~$X$ with the addition induced
by the disjoint union.

\end{enumerate}\end{dfn}

\begin{eg}\label{BMpseudo_eg}
If $f\!:M\!\lra\!X$ is a Borel-Moore $k$-pseudocycle and $Z\!\subset\!X$ is a closed subset 
of dimension at most $k\!-\!2$, then $f|_{M-f^{-1}(Z)}$ is also a Borel-Moore $k$-pseudocycle 
(with $\Bd f|_{M-f^{-1}(Z)}$ contained in~$(\Bd f)\!\cup\!Z$)  and
$$\wt{f}\!:\wt{M}\!\equiv\![0,1]\!\times\!M\!-\!\{1\}\!\times\!f^{-1}(Z)\lra X, \quad
\wt{f}(t,p)=f(p),$$
is a Borel-Moore pseudocycle equivalence between~$f$ and~$f|_{M-f^{-1}(Z)}$.
\end{eg}

\begin{thm}\label{main_thm}
Let $X$ be a smooth manifold. 
There exist homomorphisms of graded abelian groups
\BE{mainthm_e}\Psi_*\!: \Hcl_*(X;\Z)\lra\cHcl_*(X)
\quad\hbox{and}\quad
\Phi_*\!: \cHcl_*(X)\lra \Hcl_*(X;\Z)\EE
that are natural with respect to proper maps 
such that \hbox{$\Phi_*\!\circ\!\Psi_*\!=\!\Id$} and \hbox{$\Psi_*\!\circ\!\Phi_*\!=\!\Id$}.
\end{thm}

\noindent
A \sf{pseudocycle} is a Borel-Moore pseudocycle~$f$ as in Definition~\ref{BMpseudo_dfn} 
such that the closure $\ov{f(M)}$ of~$f(M)$ in~$X$ is compact.
Two pseudocycles~$f_0$ and~$f_1$ are \sf{equivalent} if there exists 
a Borel-Moore pseudocycle equivalence~$\wt{f}$ as in Definition~\ref{BMpseudo_dfn} 
such that $\ov{\wt{f}(M)}$ is a compact subset of~$X$.
The set~$\cH_k(X)$ of equivalence classes of $k$-pseudocycles to $X$ with 
the addition induced by the disjoint union is also an abelian group.
The analogue of Theorem~\ref{main_thm} for pseudocycles and the standard singular homology
is \cite[Theorem~1.1]{Z}.

\begin{rmk}\label{McDuff_rmk}
Let $X$, $f_0$, $f_1$, $\wt{f}$, $Y_0\!\subset\!M_0$, and $Y_1\!\subset\!M_1$ 
be as in Definition~\ref{BMpseudo_dfn}\ref{BMequiv_it}.
Identify a neighborhood~$W$
of~$\prt\wt{M}$ in~$\wt{M}$ with $[0,1)\!\times\!\prt\wt{M}$.
The~space
\begin{gather*}
\wh{M}\equiv\big(\wt{M}\!\sqcup\![0,1]\!\times\!(M_0\!\sqcup\!M_1)
\!-\!\{1\}\!\times\!Y_0\!-\!\{0\}\!\times\!Y_1\big)\!\big/\!\sim,
\qquad\hbox{where}\\
\wt{M}\ni p_0\sim (1,p_0)\!\in\![0,1]\!\times\!(M_0\!-\!Y_0), \quad
\wt{M}\ni p_1\sim (0,p_1)\!\in\![0,1]\!\times\!(M_1\!-\!Y_1),
\end{gather*}
is then a smooth oriented manifold with boundary $M_1\!\sqcup\!(-M_0)$.
We can deform~$\wt{F}$, while keeping it fixed on~$\prt\wt{M}$, so that it is constant
on the fibers of the projection \hbox{$W\!\lra\!\prt\wt{M}$}.
The~map
$$\wh{f}\!:\wh{M}\lra X, \qquad
\wh{f}(p)=\wt{F}(p)~~\forall\, p\!\in\!\wt{M}, \quad
\wh{f}(t,p)=f_r(p) ~~\forall~p\!\in\!M_r,\,r\!=\!0,1,$$
is then well-defined and smooth.
It satisfies the conditions in Definition~\ref{BMpseudo_dfn}\ref{BMequiv_it} 
with~$\wt{f}$ replaced by~$\wh{f}$ and $Y_0,Y_1\!=\!\eset$.
Thus, D.~McDuff's idea of attaching two collars, which is used in 
the proof of \cite[Theorem~1.1]{Z}, leads to a more relaxed, but equivalent,
formulation of pseudocycle equivalence than the traditional one, with $Y_0,Y_1\!=\!\eset$.
\end{rmk}

\begin{rmk}\label{contdfn_rmk}
As with \cite[Theorem~1.1]{Z}, it is sufficient for the purposes of Theorem~\ref{main_thm}
to require Borel-Moore pseudocycles and equivalences to be  just continuous.
All constructions in this paper would go through; Lemma~\ref{bdpush_lmm} would no longer
be needed.
On the other hand, smooth pseudocycles are more advantageous for transversality 
considerations.
\end{rmk}

\noindent
The constructions in this paper and in~\cite{Z} are direct and geometric;
both are motivated by the outline proposed in \cite[Section~7.1]{McSa}.
The proof of Theorem~\ref{main_thm} is conceptually similar to the proof of \cite[Theorem~1.1]{Z}, 
but the specifics are different because the Borel-Moore homology does not behave like 
the standard singular homology.
Inspired by~\cite{Spanier93}, we use the chain complex~$\Scl_{\{U\};*}(X;\Z)$
of singular chains that are locally finite in~$X$ and lie in a subspace $U\!\subset\!X$
to adjust the construction in~\cite{Z} to the setting of Theorem~\ref{main_thm}.
The homologies~$\Hcl_{\{U\};*}(X;\Z)$ of this complex, $\Hcl_*(X;\Z)$ of~$\Scl_*(X;\Z)$,
and $\Hcl_*(X,\{U\};\Z)$ of the quotient complex 
\BE{Sclquotdfn_e}\Scl_*\big(X,\{U\};\Z\big)\equiv \Scl_*(X;\Z)/\Scl_{\{U\};*}(X;\Z)\EE
form an exact triangle.
Given a Borel-Moore $k$-pseudocycle~$f$ to~$X$, we construct an arbitrarily small neighborhood~$U$
of $\Bd f$ with $\Hcl_{\{U\};*}(X;\Z)$ vanishing for $l\!>\!k\!-\!2$ and define
an element~$[f]_{X;U}$ in~$\Hcl_*(X,\{U\};\Z)$.
Via the aforementioned exact triangle, $[f]_{X;U}$ corresponds to an element~$[f]$ 
in~$\Hcl_*(X;\Z)$.
It is shown in~\cite{Z} that for each $k$-pseudocycle~$f$ there is an arbitrarily small
neighborhood~$U$ of $\Bd f$ with $H_l(U;\Z)$ vanishing for $l\!>\!k\!-\!2$;
a~class~$[f]_{X;U}$ is then constructed in~$H_k(X,U;\Z)$.
Our neighborhoods~$U$ are more carefully chosen versions of
the neighborhoods~$U$ of~\cite{Z}; see the proof of Proposition~\ref{neighb2_prp}.\\

\noindent
Section~\ref{outline_subs} outlines the proof of Theorem~\ref{main_thm} in
Section~\ref{mainpf_sec}.
This outline is nearly identical to \cite[Section~1.2]{Z},
with the standard homology theory replaced by an appropriate
homology theory
of locally finite singular chains.
However, care needs to be exercised in actually implementing this outline
as we are now dealing with infinite chains.
Section~\ref{BMhom_sec} thoroughly reviews the relevant background on the Borel-Moore homology
in a straightforward manner readily accessible to a broad mathematical audience
and provides the necessary tools to adapt the approach of~\cite{Z}.
In order to show that the Borel-Moore pseudocycles represent all of the cohomology of an oriented manifold,
we also give a relatively simple proof of a Poincar\'e Duality between the singular cohomology
of such a manifold and its Borel-Moore homology.
Our proof is motivated by the approach of \cite[Appendix~A]{MiSt}, which
shows that the compactly supported cohomology of an oriented manifold is
dual to its standard singular homology.
Throughout the remainder of this paper, a \sf{manifold} will always mean a smooth manifold.\\

\noindent
We are grateful to the anonymous reviewer for many helpful suggestions.

\subsection{Outline of Section~\ref{mainpf_sec}}
\label{outline_subs}

\noindent
An oriented $k$-manifold is equipped with a \sf{fundamental class}
$[M]\!\in\!\Hcl_k(M;\Z)$;
see Proposition~\ref{FC_prp}.
A smooth proper map $f\!:M\!\lra\!X$ from such a manifold determines 
an element 
$$[f]\equiv f_*[M] \in \Hcl_k(X;\Z)\,.$$
A Borel-Moore $k$-pseudocycle $f\!:M\!\lra\!X$ need not be a proper map.
However, one can choose a closed $k$-submanifold with boundary, $\ov{V}\!\subset\!M$,
so that $f|_{\ov{V}}$ is proper and $f(M\!-\!\ov{V})$ lies in a small neighborhood~$U$ of~$\Bd f$.
This implies that $f|_{\ov{V}}$ determines an element
$$[f]_{X;U}\equiv\big[f|_{\ov{V}}\big]\equiv f_*[\ov{V}]\in \Hcl_k\big(X,\{U\};\Z\big).$$
By Proposition~\ref{neighb2_prp}, $U$ can be chosen so that $\Hcl_k(X,\{U\};\Z)$
is
naturally isomorphic to~$\Hcl_k(X;\Z)$.\\

\noindent
In order to show that the image~$[f]$ of $f_*[\ov{V}]$ in $\Hcl_k(X;\Z)$ depends only on~$f$, 
we replace the chain complex~\e_ref{Sclquotdfn_e} by a quotient complex~$\ovScl_*(X,\{U\};\Z)$.
The latter is the direct adaptation of the complex~$\ov{S}_*(X)$ of~\cite{Z}
from the standard singular chains to the locally finite singular chains.
The homology~$\ovHcl_*(X,\{U\};\Z)$ of~$\ovScl_*(X,\{U\};\Z)$ is naturally isomorphic 
to~$\Hcl_*(X,\{U\};\Z)$,
but cycles and boundaries in this chain complex can be constructed more easily;
see the last paragraph of \cite[Section~2.3]{Z}.\\

\noindent
A Borel-Moore pseudocycle equivalence $\ti{f}\!:\ti{M}\!\lra\!X$ between
two Borel-Moore pseudocycles 
$$f_r\!:M_r\!\lra\!X, \qquad r\!=\!0,1,$$ 
gives rise to a chain equivalence 
between the corresponding cycles in $\ovScl_*(X,\{W\};\Z)$, 
for a small neighborhood~$W$ of~$\Bd\ti{f}$.
This implies~that 
$$[f_0]_{X;W}=[f_1]_{X;W}\in 
\ovHcl_k\big(X,\{W\};\Z\big)\cong\Hcl_k\big(X,\{W\};\Z\big).$$
By Proposition~\ref{neighb2_prp}, $W$ can be chosen so that 
$\Hcl_k(X;\Z)$ naturally injects into $\Hcl_k(X,\{W\};\Z)$.
Thus, 
$$[f_0]=[f_1]\in\Hcl_k(X;\Z)$$ 
and the homomorphism $\Phi_*$ is well-defined;
see Section~\ref{phi_sec} for details.\\

\noindent
The homomorphism $\Psi_*$ is constructed by first showing that 
all codimension~1 faces of the  simplices of a cycle
in~$\ovScl_k(X;\Z)$ come in pairs with opposite orientations;
see Lemma~\ref{gluing_l1}.
By gluing the $k$-simplices along the codimension~1 faces paired up in this way,
we obtain a proper map from a simplicial complex~$M'$ to~$X$.
The complement of the codimension~2 simplices in~$M'$ is a manifold;
the continuous map from~it can be smoothed out in a standardized manner via Lemma~\ref{bdpush_lmm}.
This systematic procedure produces a Borel-Moore pseudocycle from a cycle in~$\ovScl_k(X;\Z)$.
A chain equivalence between two $k$-cycles in~$\ovScl_k(X;\Z)$, $\{c_0\}$ and~$\{c_1\}$,
similarly determines a Borel-Moore pseudocycle equivalence between 
the pseudocycles obtained from~$\{c_0\}$ and~$\{c_1\}$.\\

\noindent
In Section~\ref{isom_sec}, we conclude by confirming that the homomorphisms $\Psi_*$ 
and $\Phi_*$ are mutual inverses.
As in~\cite{Z}, it is fairly straightforward to show  
that the map $\Phi_*\!\circ\!\Psi_*$ is the identity on~$\Hcl_*(X;\Z)$.
Following the approach in~\cite{Z},
we then show that the homomorphism~$\Phi_*$ is injective.\\

\noindent
We now note some basic facts concerning proper maps
that will be used in the proof of Theorem~\ref{main_thm}.

\begin{lmm}\label{proper_lmm}
Let $f\!:M\!\lra\!X$ be a continuous map.
\begin{enumerate}[label=(\arabic*),leftmargin=*]

\item\label{proper1} If $X$ is Hausdorff and \hbox{$U\!\subset\!X$} is an open neighborhood of $\Bd f$,
then $f|_{M-f^{-1}(U)}$ is a proper map.

\item\label{Bdext_it} If $X$ is Hausdorff and locally compact, then
$$\Bd f|_{M-B}\subset (\Bd f)\!\cup\ov{f(B)} \qquad\forall~B\!\subset\!M\,.$$

\item If $f$ is proper,  $B\!\subset\!M$ is closed, and either $M$ or $X$ is Hausdorff, 
then $f|_B$ is also proper.

\item\label{proper4} If $f$ is proper and $X$ is Hausdorff and locally compact,
then $f$ is a closed map.

\item\label{properext_it} 
If $X$ is Hausdorff and admits a locally finite cover $\{A_i\}_{i\in\cI}$ by compact subsets, 
$M$ is normal and locally compact, and $B\!\subset\!M$ is a closed subset
such that $f|_B$ is proper, then there exists an open neighborhood~$W\!\subset\!M$ of~$B$ so
that $f|_{\ov{W}}$ is still proper.

\end{enumerate}
\end{lmm}

\begin{proof}
\begin{enumerate}[label=(\arabic*),leftmargin=*]
\item Let $A \subset X$ be a compact set disjoint from $\Bd f$. Then, sets of the form $X - \overline{f(M-K)}$, where $K \subset M$ is compact, form an open cover of $A$. Taking a finite subcover yields a compact set $K' \subset M$ such that $A$ is disjoint from $\overline{f(M-K')}$. Therefore, $f^{-1}(A)$ is a closed subset of the compact set $K'$, and is thus compact.

\item We prove the contrapositive. Let $x \in X$ be disjoint from $\Bd f \cup \overline{f(B)}$. As this latter set is closed, there is a precompact open neighborhood $U$ of $x$ that is disjoint from it. As $X$ is Hausdorff and $\overline{U}$ is compact, we may, by shrinking $U$, assume $\overline{U}$ is disjoint from $\Bd f \cup \overline{f(B)}$. By Lemma~\ref{proper_lmm}\ref{proper1}, $f^{-1}(\overline{U})$ is compact. It is also disjoint from $B$. Therefore, as $x \not\in \overline{f(M - B - \overline{U})}$, $x\not\in \Bd f|_{M-B}$

\item Let $A \subset X$ be compact. As either $M$ or $X$ is Hausdorff and $f$ is proper, then $f^{-1}(A)$ is a closed compact set. As $B$ is closed, $f^{-1}(A) \cap B$ is a closed subset of a compact set, so it is compact.

\item Let $B \subset M$ be closed and $x \in X$ be a limit point of $f(B)$. Let $U$ be a precompact open neighborhood of $x$. As $f$ is proper, $f^{-1}(\overline{U})$ is compact. As $B$ is closed, $f^{-1}(\overline{U}) \cap B$ is compact. As $f$ is continuous, $\overline{U} \cap f(B)$ is compact and therefore closed. As $x$ is a limit point of $\overline{U} \cap f(B)$, it lies in $f(B)$.

\item Since the cover $\{A_i\}_{i\in\cI}$ of~$X$ is locally finite, every compact subset~$A\!\subset\!X$
is covered by finitely many elements of this collection.
It is thus sufficient to construct a neighborhood~$W\!\subset\!M$ of~$B$
so that \hbox{$\ov{W}\!\cap\!f^{-1}(A_i)$} is compact for every $i\!\in\!\cI$.\\

\noindent
The cover $\{f^{-1}(A_i)\!\}_{i\in\cI}$ of~$M$ is locally finite and 
consists of closed subsets of~$M$.
For each $i\!\in\cI$, let
$$\cI_i=\big\{j\!\in\!\cI\!:A_i\!\cap\!A_j\!\neq\!\eset\big\} \quad\hbox{and}\quad
B_i^c=\bigcup_{j\in\cI-\cI_i}\!\!\!\!\!f^{-1}(A_j)\subset M.$$
By the compactness of~$A_i$, the collection $\cI_i$ is finite.
Since $\{f^{-1}(A_i)\!\}_{i\in\cI-\cI_i}$ is a locally finite collection of closed subsets
of~$M$, $B_i^c$ is a closed subset of~$M$ disjoint from the closed subset~$f^{-1}(A_i)$.
Let $U_i\!\subset\!M$ be an open neighborhood of~$f^{-1}(A_i)$ disjoint from~$B_i^c$.
Since
$$\big\{j\!\in\!\cI\!:U_i\!\cap\!U_j\!\neq\!\eset\big\}
\subset\bigcup_{k\in\cI_i}\!\cI_k,$$
the open cover $\{U_i\}_{i\in\cI}$ is locally finite.\\

\noindent
For each $i\!\in\!\cI$, $B\!\cap\!f^{-1}(A_i)\!\subset\!M$ is a compact subset.
Let $V_i\!\subset\!M$ be an open neighborhood of $B\!\cap\!f^{-1}(A_i)$ so that 
$\ov{V}_i\!\subset\!M$ is compact and contained in~$U_i$.
Let 
$$W=\bigcup_{i\in\cI}\!V_i\subset M.$$
Since the collection $\{\ov{V}_i\}_{i\in\cI}$ is locally finite,
$$\ov{W}=\bigcup_{i\in\cI}\!\ov{V}_i\subset M.$$
For any $i\!\in\!\cI$,
$$\ov{W}\!\cap\!f^{-1}(A_i)=\bigcup_{j\in\cI_i}\!(\ov{V}_j\!\cap\!f^{-1}(A_i)\!\big)\subset M.$$
The above finite union of compact subsets of $M$ is compact, as needed.
\end{enumerate}
\end{proof}

\section{Borel-Moore homology}
\label{BMhom_sec}

\subsection{Standard simplicies}
\label{notat_subs}

\noindent
In order to set up notation for the standard simplices, their subsets, and maps between
them consistent with~\cite{Z}, we reproduce most of \cite[Section~2.1]{Z}.
The present section can be skipped at first and referred to as needed later. Throughout, $\Z^{\ge0}$ denotes the nonnegative integers and $\Z^+$ denotes the positive integers.\\

\noindent
For $k\!\in\!\Z^{\ge0}$, let
$$[k]=\big\{0,1,2,\ldots,k\big\}.$$
For a finite subset $A\!\subset\!\R^k$, we denote by $\CH(A)$ and $\CH^0(A)$ 
the (\sf{closed}) \sf{convex hull} of~$A$ and the \sf{open convex hull} of~$A$, 
respectively, i.e.
\begin{equation*}\begin{split}
\CH(A)&=\Big\{\sum_{v\in A}t_vv\!: t_v\!\in\![0,1];\,\sum_{v\in A}t_v\!=\!1\Big\}
\qquad\hbox{and}\\
\CH^0(A)&=\Big\{\sum_{v\in A}t_vv\!: t_v\!\in\!(0,1);\,\sum_{v\in A}t_v\!=\!1\Big\}.
\end{split}\end{equation*}
If $B\!\subset\!\R^m$ is also a finite set, a map $f\!:\CH(A)\!\lra\!\CH(B)$ is
\sf{linear} if
$$f\bigg(\sum_{v\in A}t_vv\bigg)=\sum_{v\in A}t_vf(v)
\qquad\forall~t_v\!\in\![0,1]^A~~\hbox{s.t.}~~\sum_{v\in A}t_v\!=\!1.$$
Such a map is determined by its values on~$A$.\\

\noindent
For each $p\!=\!1,\ldots,k$, let $e_p$ be the $p$-th coordinate vector in $\R^k$.
Put $e_0\!=\!0\!\in\!\R^k$.
Denote by
$$\De^k=\CH\big(e_0,e_1,\ldots,e_k\big) \qquad\hbox{and}\qquad
\Int\De^k=\CH^0\big(e_0,e_1,\ldots,e_k\big)$$
the \sf{standard $k$-simplex} and its interior.
Let 
$$b_k=\frac{1}{k\!+\!1}\bigg(\sum_{p=0}^{p=k}e_p\bigg)=
\Big(\frac{1}{k\!+\!1},\ldots,\frac{1}{k\!+\!1}\Big)\in\Bbb{R}^k$$
be the \sf{barycenter} of~$\De^k$.\\

\noindent
For each $p\!=\!0,1,\ldots,k$, let
$$\De^k_p=\CH\big(\big\{e_q\!: q\!\in\![k]\!-\!\{p\}\!\big\}\big)
\qquad\hbox{and}\qquad
\Int\De_p^k=\CH^0\big(\big\{e_q\!: q\!\in\![k]\!-\!\{p\}\!\big\}\big)$$
denote the \sf{$p$-th face} of $\De^k$ and its interior.
Define a linear~map
$$\io_{k;p}\!: \De^{k-1}\lra\De^k_p\subset\De^k \qquad\hbox{by}\qquad
\io_{k;p}(e_q)=\begin{cases}
e_q,& \hbox{if}~q\!<\!p;\\
e_{q+1},& \hbox{if}~q\!\ge\!p.
\end{cases}$$
We also define a projection map
$$\wt\pi^k_p\!: \De^k\!-\!\{e_p\}\lra\De^k_p \qquad\hbox{by}\qquad
\wt\pi^k_p\Big(\sum_{q=0}^{q=k}t_qe_q\Big)=
\frac{1}{1\!-\!t_p}\bigg(
\sum_{\begin{subarray}{c}0\le q\le k\\ q\neq p\end{subarray}}\!\!\!t_qe_q\bigg).$$
Put
$$b_{k;p}=\io_{k;p}(b_{k-1}),\qquad
b_{k;p}'=\frac{1}{k\!+\!1}\bigg(b_k+\!
\sum_{\begin{subarray}{c}0\le q\le k\\ q\neq p\end{subarray}}\!\!\!e_q\bigg).$$
The points $b_{k;p}$ and $b_{k;p}'$ are the barycenters of 
the $(k\!-\!1)$-simplex~$\De^k_p$
and of the $k$-simplex spanned by~$b_k$ and the vertices of~$\De^k_p$.
Define a neighborhood of $\Int\De^k_p$ in $\De^k$ by
\begin{equation*}\begin{split}
U^k_p&=\big\{t_pb_{k;p}'\!+\!
\sum_{\begin{subarray}{c}0\le q\le k\\ q\neq p\end{subarray}}\!\!\!t_qe_q\!: 
t_p\!\ge\!0,~t_q\!>\!0~\forall q\!\neq\!p;~ \sum_{q=0}^{q=k}t_q\!=\!1\big\}\\
&=\big(\Int\De^k_p\big)\!\cup\!
\CH^0\big(\big\{e_q\!:q\!\in\![k]\!-\!\{p\}\!\big\}\!\cup\!\{b_{k;p}'\}\!\big);
\end{split}\end{equation*}
see Figure~\ref{simplices_fig}.
These disjoint neighborhoods are used to construct Borel-Moore pseudocycles out of Borel-Moore homology cycles
via Lemma~\ref{bdpush_lmm}.\\
\begin{figure}
\begin{pspicture}(-.7,-1.8)(10,1.5)
\psset{unit=.4cm}
\psline{->}(0,-3)(8,-3)\psline{->}(1,-4)(1,4)\psline(1,3)(7,-3)
\pscircle*(1,-3){.2}\pscircle*(7,-3){.2}\pscircle*(1,3){.2}
\rput(.4,-3.5){\smsize{$e_0$}}\rput(7.4,-3.5){\smsize{$e_1$}}
\rput(.3,2.7){\smsize{$e_2$}}
\rput(4,-3.7){\smsize{$\De^2_2$}}\rput{90}(.3,0){\smsize{$\De^2_1$}}
\rput{-45}(5,0){\smsize{$\De^2_0$}}
\pscircle*(3,-1){.2}\pscircle*(3.67,-2.33){.2}
\rput(2.4,-1){\smsize{$b_2$}}\rput(4.6,-1.9){\smsize{$b_{2;2}'$}}
\psline{->}(15,-3)(23,-3)\psline{->}(16,-4)(16,4)\psline(16,3)(22,-3)
\pscircle*(16,-3){.2}\pscircle*(22,-3){.2}\pscircle*(16,3){.2}
\rput(15.4,-3.5){\smsize{$e_0$}}\rput(22.4,-3.5){\smsize{$e_1$}}
\rput(15.3,2.7){\smsize{$e_2$}}\rput(19,-3.7){\smsize{$\De^2_2$}}
\pscircle*(16,0){.2}\pscircle*(19,0){.2}
\rput(15.1,-.3){\smsize{$b_{2;1}$}}\rput(20.2,0){\smsize{$b_{2;0}$}}
\psline{->}(30,-3)(38,-3)\psline{->}(31,-4)(31,4)\psline(31,3)(37,-3)
\pscircle*(31,-3){.2}\pscircle*(37,-3){.2}\pscircle*(31,3){.2}
\rput(30.4,-3.5){\smsize{$e_0$}}\rput(37.4,-3.5){\smsize{$e_1$}}
\rput(30.3,2.7){\smsize{$e_2$}}
\rput(34,-3.7){\smsize{$\De^2_2$}}
\psline[linestyle=dotted](31,-3)(33.67,-2.33)\psline[linestyle=dotted](37,-3)(33.67,-2.33)
\pnode(33.5,-2.65){A}\rput(29,0){\rnode{B}{\smsize{$U_2^2$}}}
\nccurve[nodesep=0,angleA=-90,angleB=180,ncurv=.5,linewidth=.02]{->}{B}{A}
\end{pspicture}
\caption{The standard 2-simplex and some of its distinguished subsets}
\label{simplices_fig}
\end{figure}
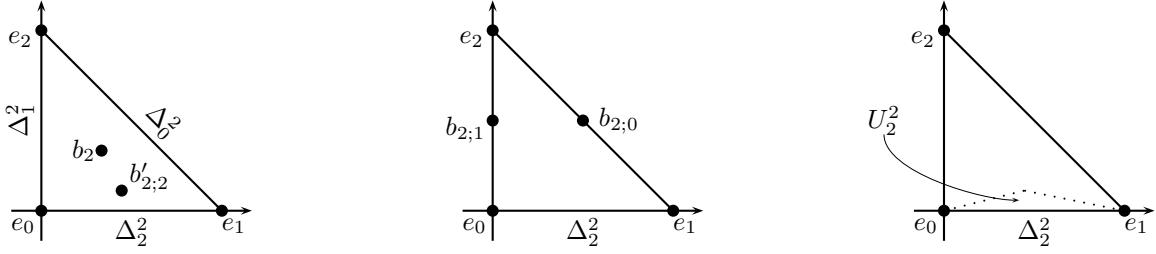

\noindent
If $p,q\!=\!0,1,\ldots,k$ and $p\!\neq\!q$, let 
$$\De^k_{p,q}\equiv \De^k_p\!\cap\!\De^k_q$$
be the corresponding codimension~2 simplex.
Define a projection map
$$\wt\pi^k_{p,q}\!: \De^k-\CH(e_p,e_q)\lra\De^k_{p,q} \qquad\hbox{by}\qquad
\wt\pi^k_{p,q}\Big(\sum_{r=0}^{r=k}t_re_r\Big)=
\frac{1}{1\!-\!t_p\!-\!t_q}\bigg(
\sum_{\begin{subarray}{c}0\le r\le k\\ r\neq p,q\end{subarray}}\!\!\!\!t_re_r\bigg).$$
We define a neighborhood of $\Int\De^k_{p,q}$ in $\De^k$ by
\begin{equation*}\begin{split}
U^k_{p,q}&=\big\{t_p\io_{k;p}(b_{k-1;\io_{k;p}^{-1}(q)}')
\!+\!t_q\io_{k;q}(b_{k-1;\io_{k;q}^{-1}(p)}')
\!+\!\sum_{\begin{subarray}{c}0\le r\le k\\ r\neq p,q\end{subarray}}\!\!\!t_re_r\!:
t_p,t_q\!\ge\!0,\, t_r\!>\!0\, \forall r\!\neq\!p,q;\,\sum_{r=0}^{r=k}t_r\!=\!1\big\}\\
&=\big(\Int\De^k_{p,q}\big)\!\cup\!
\CH^0\big(\big\{e_r\!:r\!\in\![k]\!-\!\{p,q\}\!\big\}\!\cup\!
\{\io_{k;p}(b_{k-1;\io_{k;p}^{-1}(q)}'),\io_{k;q}(b_{k-1;\io_{k;q}^{-1}(p)}')\}\!\big);
\end{split}\end{equation*}
see Figure~\ref{simplices_fig2}.
These disjoint neighborhoods are used to construct Borel-Moore pseudocycle equivalences out of 
Borel-Moore bounding chains via Lemma~\ref{bdpush_lmm}.\\

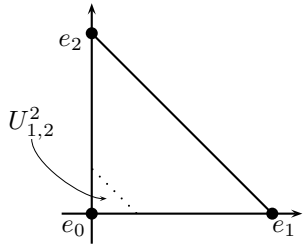
\begin{figure}
\begin{pspicture}(-5,-1.8)(10,1.5)
\psset{unit=.4cm}
\psline{->}(2,-3)(10,-3)\psline{->}(3,-4)(3,4)\psline(3,3)(9,-3)
\pscircle*(3,-3){.2}\pscircle*(9,-3){.2}\pscircle*(3,3){.2}
\rput(2.4,-3.5){\smsize{$e_0$}}\rput(9.4,-3.5){\smsize{$e_1$}}
\rput(2.3,2.7){\smsize{$e_2$}}
\psline[linestyle=dotted](3,-1.5)(4.5,-3)
\pnode(3.5,-2.5){A}\rput(1,0){\rnode{B}{\smsize{$U^2_{1,2}$}}}
\nccurve[nodesep=0,angleA=-90,angleB=180,ncurv=.5,linewidth=.02]{->}{B}{A}
\end{pspicture}
\caption{The standard 2-simplex and a distinguished neighborhood of a codimension 2 simplex}
\label{simplices_fig2}
\end{figure}
\noindent
Denote by $\cS_k$  the group of permutations of the set~$[k]$. 
We view the set $\cS_k$ as a subset of $\cS_{k+1}$ by setting  
\hbox{$\tau(k\!+\!1)\!=\!k\!+\!1$} for each \hbox{$\tau\!\in\!\cS_k$}. 
For any $\tau\!\in\!\cS_k$, let
$$\tau\!:\De^k\lra\De^k$$
be the linear map defined by
$$\tau(e_q)=e_{\tau(q)} \qquad\forall\,q=0,1,\ldots,k.$$

\begin{lmm}[{\cite[Lemma~2.1]{Z}}]\label{bdpush_lmm}
Let $k\!\in\!\Z^+$, $Y\!\subset\!\De^k\!$ be the $(k\!-\!2)$-skeleton of $\De^k$, and 
\hbox{$\wt{Y}\!\subset\!\De^{k+1}$} be the $(k\!-\!2)$-skeleton of $\De^{k+1}$.
There exist continuous functions 
$$\vp_k\!: \De^k\lra\De^k \qquad\hbox{and}\qquad 
\wt\vp_{k+1}\!: \De^{k+1}\lra\De^{k+1}$$
such~that
\begin{enumerate}[label=(\alph*),leftmargin=*]
\item $\vp_k$ is smooth outside of $Y$ and 
$\wt\vp_{k+1}$ is smooth outside of $\wt{Y}$;
\item for all $p\!=\!0,\ldots,k$ and $\tau\!\in\!\cS_k$,
\BE{bdpush_e1}
\vp_k|_{U^k_p}=\wt\pi^k_p\big|_{U^k_p}\,, \qquad
\vp_k\!\circ\!\tau=\tau\!\circ\!\vp_k;\EE
\item for all $p,q\!=\!0,\ldots,k\!+\!1$ with $p\!\neq\!q$ and 
$\wt\tau\!\in\!\cS_{k+1}$,
\BE{bdpush_e2}
\wt\vp_{k+1}|_{U^{k+1}_{p,q}}=\wt\pi^{k+1}_{p,q}\big|_{U^{k+1}_{p,q}}, \quad
\wt\vp_{k+1}\!\circ\!\wt\tau=\wt\tau\!\circ\!\ti\vp_{k+1}, \quad
\wt\vp_{k+1}\!\circ\!\io_{k+1;p}=\io_{k+1;p}\!\circ\!\vp_k.\EE
\end{enumerate}
\end{lmm}

\subsection{Basic definitions}
\label{BMdfn_subs}

\noindent
Let $R$ be a commutative ring with unity~1 and $X$ be a topological space.
For $k\!\in\!\Z^{\ge0}$, denote by $\Hom(\De^k,X)$ the set of 
\sf{singular $k$-simplicies} on~$X$, i.e.~of continuous maps from~$\De^k$ to~$X$.
A
\sf{singular chain on~$X$ with coefficients in~$R$}, i.e.~a map
\BE{cmapdfn_e} c\!:\Hom(\De^k,X)\lra R,\EE
can be written (in a slight abuse of notation)
as a formal sum
\BE{BMch_e} c=\sum_{\si\in\Hom(\De^k,X)}\!\!\!\!\!\!\!\!\!\!\!\!a_{\si}\si~, 
\qquad a_{\si}\in R\,.\EE
We identify $\Hom(\De^k,X)$ with a subset of such maps by defining
$$\si\!: \Hom(\De^k,X)\lra R, \qquad 
\si(\tau)=\begin{cases}1,&\hbox{if}~\tau\!=\!\si;\\
0,&\hbox{if}~\tau\!\neq\!\si;\end{cases} \quad
\forall\,\si,\tau\!\in\!\Hom(\De^k,X).$$
We say that a singular $k$-simplex~$\si$ \sf{appears} in a singular chain~$c$ as 
in~\e_ref{cmapdfn_e} and~\e_ref{BMch_e} if $c(\si)\!\equiv\!a_{\si}$
is not zero.\\

\noindent
For a singular chain~$c$ as in~\e_ref{cmapdfn_e} and~\e_ref{BMch_e}, define
\BE{suppdfn_e}\supp(c)=
\bigcup_{\begin{subarray}{c}\si\in\Hom(\De^k,X)\\ c(\si)\neq0\end{subarray}}
\!\!\!\!\!\!\!\!\!\!\!\!\si(\De^k)
=\bigcup_{\begin{subarray}{c}\si\in\Hom(\De^k,X)\\ a_{\si}\neq0\end{subarray}}
\!\!\!\!\!\!\!\!\!\!\!\!\si(\De^k)\subset X\EE
to be the \sf{support of~$c$}.
If a $k$-simplex~$\si$ appears in~$c$, then $\si(\De^k)\!\subset\!\supp(c)$.
For  $U\!\subset\!X$, let
\BE{alecUdfn_e}\begin{split}
\ale_c(U)
&= \big\{\si\!\in\!\Hom(\De^k,X)\!:c(\si)\!\neq\!0,\,
\si(\De^k)\!\cap\!U\neq\!\eset\big\}\\
&=\big\{\si\!\in\!\Hom(\De^k,X)\!:a_{\si}\!\neq\!0,\,
\si(\De^k)\!\cap\!U\neq\!\eset\big\}.
\end{split}\EE
A \sf{finite singular $k$-chain on~$X$ with coefficients in~$R$} 
is a~map~$c$ as in~\e_ref{cmapdfn_e} such that the~set $\ale_c(X)$ is finite.
The $R$-module of such chains is the $k$-th module of 
the usual chain complex~$S_*(X;R)$ determining the standard singular homology~$H_*(X;R)$
of~$X$.\\

\noindent
A \sf{Borel-Moore $k$-chain on~$X$} is a~map~$c$ as in~\e_ref{cmapdfn_e}
such that for every $x\!\in\!X$ there exists an open neighborhood \hbox{$U_x\!\subset\!X$} of~$x$
so that the set $\ale_c(U_x)$ defined by~\e_ref{alecUdfn_e} is finite.
If $X$ is second countable, at most countably many simplicies appear 
in a Borel-Moore $k$-chain on~$X$.
If $X$ is Hausdorff, the support~\e_ref{suppdfn_e} of a Borel-Moore chain~$c$ 
is closed in~$X$.
The set~$\Scl_k(X;R)$ of Borel-Moore $k$-chains on~$X$ with coefficients in~$R$
is an $R$-module under the addition and scalar multiplication
of the values of the chains on the $k$-simplices.
This set contains~$\Hom(\De^k,X)$.
We call a~map
\BE{hrigid_e}h\!:\Hom(\De^k,X)\lra \Scl_p(X;R)\EE
\sf{rigid} if
\BE{hrigid_e2}
\supp\big(h(\si)\!\big)\subset\si(\De^k) \qquad\forall~\si\!\in\!\Hom(\De^k,X).\EE
Rigid maps, like the acyclic carriers in \cite[Section 13]{Mu2}, greatly ease the construction of chain homotopies.

\begin{lmm}\label{hrigid_lmm}
Let $X$ be a topological space. 
A rigid map~$h$ as in~\e_ref{hrigid_e} induces a homomorphism
\BE{hrigid_e3}\begin{split}
&\hspace{1in}h\!:\Scl_k(X;R)\lra\Scl_p(X;R),\\
&\big\{h(c)\!\big\}(\tau)= \sum_{\si\in\Hom(\De^k,X)}\!\!\!\!\!\!\!\!\!\!\!\!
c(\si)\big\{\!h(\si)\!\big\}(\tau)\in R
\quad\forall~\tau\!\in\!\Hom(\De^p,X),~c\!\in\!\Scl_k(X;R),
\end{split}\EE
extending~\e_ref{hrigid_e} such that 
\BE{hrigid_e5} 
\supp\big(h(c)\!\big)\subset\supp(c)\qquad\forall~c\!\in\!\Scl_k(X;R).\EE
\end{lmm}

\begin{proof} Let $c\!\in\!\Scl_k(X;R)$ and $\tau\!\in\!\Hom(\De^p,X)$.
By the compactness of $\tau(\De^p)\!\subset\!X$,
there exists an open neighborhood~$U_{\tau}$ of~$\tau(\De^p)$ in~$X$ such that
the~set $\ale_c(U_{\tau})$ is finite.
By~\e_ref{suppdfn_e} and~\e_ref{hrigid_e2},
\begin{equation*}\begin{split}
\big\{\si\!\in\!\Hom(\De^k,X)\!:c(\si)\big\{\!h(\si)\!\big\}(\tau)\!\neq\!0\big\}
&\subset \big\{\si\!\in\!\Hom(\De^k,X)\!:c(\si)\!\neq\!0,\,
\tau(\De^p)\!\subset\!\supp\big(h(\si)\big)\big\}\\
&\subset \big\{\si\!\in\!\Hom(\De^k,X)\!:c(\si)\!\neq\!0,\,
\tau(\De^p)\!\subset\!\si(\De^k)\big\}\\
&\subset \big\{\si\!\in\!\Hom(\De^k,X)\!:c(\si)\!\neq\!0,\,
\si(\De^k)\!\cap\!U_{\tau}\!\neq\!\eset\big\}=\ale_c(U_{\tau}).
\end{split}\end{equation*}
Thus, the sum in~\e_ref{hrigid_e3} is finite.\\

\noindent
Let $c\!\in\!\Scl_k(X;\Z)$, $x\!\in\!X$, and $U_c$ be an open neighborhood
of~$x$ in~$X$ such that the~set $\ale_c(U_c)$ is finite.
For each $\si\!\in\!\Hom(\De^k,X)$, let $U_{\si}$ 
be an open neighborhood of~$x$ in~$X$ such that
the~set 
$$\ale_{h(\si)}(U_{\si})\equiv \big\{\tau\!\in\!\Hom(\De^p,X)\!:
\big\{\!h(\si)\!\big\}(\tau)\!\neq\!0,\,\tau(\De^p)\!\cap\!U_{\si}\neq\!\eset\big\}$$ 
is finite.
The~subset
$$U_x\equiv U_c \cap \bigcap_{\si\in\ale_c(U_c)}\!\!\!\!\!\!\!U_{\si}~\subset X$$
is also an open neighborhood of~$x$ in~$X$.
By~\e_ref{hrigid_e2},
$$\ale_{h(\si)}(U_c)\subset\big\{\tau\!\in\!\Hom(\De^p,X)\!:
\tau(\De^p)\!\subset\!\si(\De^k),\,\tau(\De^p)\!\cap\!U_c\!\neq\!\eset\big\}
=\eset\quad\forall~\si\!\in\!\ale_c(X)\!-\!\ale_c(U_c).$$ 
Combining this with~\e_ref{hrigid_e3}, we obtain
\begin{equation*}
\ale_{h(c)}(U_x)\subset
\bigcup_{\si\in\ale_c(X)}\!\!\!\!\!\!\ale_{h(\si)}(U_x)
=\bigcup_{\si\in\ale_c(U_c)}\!\!\!\!\!\!\!\ale_{h(\si)}(U_x)
\subset\bigcup_{\si\in\ale_c(U_c)}\!\!\!\!\!\!\!\ale_{h(\si)}(U_{\si}).
\end{equation*}
Since the last set above is finite, we conclude that $h(c)\!\in\!\Scl_p(X;R)$.\\

\noindent
It is immediate that the map~$h$ in~\e_ref{hrigid_e3} is a homomorphism of $R$-modules
and restricts to~\e_ref{hrigid_e}.
By~\e_ref{hrigid_e3} and~\e_ref{hrigid_e2},
\begin{equation*}\begin{split}
&\big\{\tau\!\in\!\Hom(\De^p,X)\!:\big\{\!h(c)\!\big\}(\tau)\!\neq\!0\big\}
\subset\bigcup_{\si\in\ale_c(X)}\!\!\!\!\!\!
\big\{\tau\!\in\!\Hom(\De^p,X)\!:\big\{\!h(\si)\!\big\}(\tau)\!\neq\!0\big\}\\
&\qquad\subset \!\!\!
\bigcup_{\begin{subarray}{c}\si\in\Hom(\De^k,X)\\ \si(\De^k)\subset\supp(c)\end{subarray}}
\hspace{-.32in}
\big\{\tau\!\in\!\Hom(\De^p,X)\!:\tau(\De^p)\!\subset\!\supp\big(h(\si)\big)\big\}
\subset 
\big\{\tau\!\in\!\Hom(\De^p,X)\!:\tau(\De^p)\!\subset\!\supp(c)\big\}.
\end{split}\end{equation*}
This establishes~\e_ref{hrigid_e5}.
\end{proof}

\noindent
In the notation of~\e_ref{BMch_e}, 
$$h(c)=\sum_{\si\in\Hom(\De^k,X)}\!\!\!\!\!\!\!\!\!\!\!\!a_{\si}h(\si)\,.$$
Each $h(\si)$ is a formal sum.
By the first part of the proof of Lemma~\ref{hrigid_lmm},
each $p$-simplex~$\tau$ appears in only finitely many chains~$h(\si)$.
Thus, the implicit double sum above can be reduced to a single sum as in~\e_ref{BMch_e}.
By the second part of the proof of Lemma~\ref{hrigid_lmm}, $h(c)$ satisfies 
the required local finiteness condition.\\

\noindent
A map
\BE{hmap_e2}\hb\!:\Hom(\De^k,X)\lra \Scl_p(\De^k;R)\!=\!S_p(\De^k;R)\EE
induces a rigid map
\BE{hmap_e2b}h\!:\Hom(\De^k,X)\lra S_p(X;R), \qquad 
h(\si)=\si_{\#}\big(\hb(\si)\!\big),\EE
and thus a homomorphism
$$h\!\equiv\!\hb_{\#}\!: \Scl_k(X;R)\lra \Scl_p(X;R).$$

\vspace{.2in}

\noindent
If $k\!\in\!\Z^+$, the \sf{boundary homomorphism} 
\BE{prtBMdfn_e}\prt_X\!: \Scl_k(X;R)\lra \Scl_{k-1}(X;R), \quad
\prt_X\!\!\!\!\!\!\!\!\!\sum_{\si\in\hbox{Hom}(\De^k,X)}\!\!\!\!\!\!\!\!\!\!\!\!\!a_{\si}\si
=\sum_{\si\in\hbox{Hom}(\De^k,X)}\sum_{p=0}^k(-1)^pa_{\si}\big(\si\!\circ\!\io_{k;p}\big)\EE
is induced by the constant map
$$\hb\!:\Hom(\De^k,X)\lra \Scl_{k-1}(\De^k;R), \quad
\hb(\si)=\prt_{\De^k}\id_{\De^k}\equiv\sum_{p=0}^k(-1)^p\io_{k;p}\,.$$
By Lemma~\ref{hrigid_lmm}, the homomorphism~\e_ref{prtBMdfn_e} is thus well-defined.
We define~$\prt_X$ on $\Scl_0(X;R)$ to be the zero homomorphism.
It is immediate that $\prt_X^2\!=\!0$. 
The quotient
$$\Hcl_k(X;R)=\frac{\ker(\prt_X\!:\Scl_k(X;R)\lra \Scl_{k-1}(X;R))}
{\Im(\prt_X\!:\Scl_{k+1}(X;R)\lra \Scl_k(X;R))}$$
is the \sf{$k$-th Borel-Moore homology module of~$X$ with coefficients in~$R$}.
If $X$ is compact, $(\Scl_*(X;\R),\prt_X)$ is the usual singular chain complex 
$(S_*(X;R),\prt_X)$ and the Borel-Moore homology modules are the standard homology modules
with coefficients in~$R$.\\

\noindent
For $q\!\in\!\Z^{\ge0}$, let 
$$S^q(X;R)\equiv\Hom_{\R}\big(S_q(X;R),R\big)$$ 
denote the usual $R$-module of the $R$-valued $p$-cochains on~$X$.
For each $\al\!\in\!S^q(X;R)$, the~map
$$\al\cap\!:\Hom(\De^{p+q},X)\lra \Scl_p(X;R), \quad
\al\!\cap\!\si=\al({}\si^q)\,{}^{p\!}\si,$$
where ${}^{p\!}\si = \si(\CH(e_0, \dots ,e_p))$ and $\si^q = \si(\CH(e_{p+1}, \dots ,e_{p+q}))$ are the $p$-th front and $q$-th back faces,
respectively,
of a singular  \hbox{$(p\!+\!q)$-simplex~$\si$}, is rigid.
By Lemma~\ref{hrigid_lmm}, this map thus induces a homomorphism
\BE{Scapdfn_e}\cap\!: S^q(X;R)\!\otimes_R\!\Scl_{p+q}(X;R)\lra \Scl_p(X;R),
\quad \al\!\otimes\!\mu\mapsto\al\!\cap\!\mu\,.\EE
This \sf{cap product} restricts to the cap product on $S^q(X;R)\!\otimes_R\!S_{p+q}(X;R)$
in the standard singular theory defined in \cite[Section~66]{Mu2}.
The homomorphism~\e_ref{Scapdfn_e} satisfies
\BE{capbd_e} \prt_X\big(\al\!\cap\!\mu\big)=(-1)^p(\de_X\al)\!\cap\!\mu
+\al\!\cap\!(\prt_X\mu)\qquad
\forall\,\al\!\in\!S^p(X;R),\,\mu\!\in\!\Scl_{p+q}(X;R),\EE
where $\de_X\!=\!\prt_X^{\,*}$.
Thus, \e_ref{Scapdfn_e} descends to a homomorphism
$$\cap\!: H^q(X;R)\!\otimes_R\!\Hcl_{p+q}(X;R)\lra \Hcl_p(X;R).$$

\subsection{Basic properties}
\label{BMprp_subs}

\noindent
Let $X$ be a topological space.
We call a collection of maps
\BE{hbrigid_e} \hb\!:\Hom(\De^k,X)\lra S_*(\De^k;R),
\quad k\!\in\!\Z^{\ge0},\EE
a \sf{pre-chain map} if 
\BE{hbrigid_e4}\prt_{\De^k}\big(\hb(\si)\!\big)
=\sum_{p=0}^k(-1)^p\big\{\io_{k;p}\big\}_{\#}\big(\hb(\si\!\circ\!\io_{k;p})\!\big)
\quad\forall~\si\!\in\!\Hom(\De^k,X),~k\!\in\!\Z^{\ge0}.\EE
A pre-chain map~$\hb$ determines a chain map
\BE{hbrigid_e4b}\hb_{\#}\!:\Scl_*(X;R)\lra\Scl_*(X;R),\EE
not necessarily preserving the grading, via \e_ref{hmap_e2b} and Lemma~\ref{hrigid_lmm}.
A linear combination of pre-chain maps is a pre-chain map.\\

\noindent
Let $\hb$ be a collection of maps as in~\e_ref{hbrigid_e}.
A~\sf{null-homotopy for~$\hb$} is a~collection of maps
$$D_{\hb}\!:\Hom(\De^k,X)\lra S_{*+1}(\De^k;R), \quad k\!\in\!\Z^{\ge0},$$
such~that
\BE{hbrigid_e7}
\prt_{\De^k}\big(D_{\hb}(\si)\!\big)=\hb(\si)-
\sum_{p=0}^k(-1)^p\big\{\io_{k;p}\big\}_{\#}
\big(D_{\hb}(\si\!\circ\!\io_{k;p})\!\big)
\quad\forall~\si\!\in\!\Hom(\De^k,X),~k\!\in\!\Z^{\ge0}.\EE
In such a case,
$$\hb_{\#}=\prt_XD_{\hb\,\#}+D_{\hb\,\#}\prt_X\!: \Scl_*(X;R)\lra\Scl_*(X;R),$$
i.e.~$D_{\hb\#}$ is a chain homotopy from~$\hb_{\#}$ to the zero homomorphism.

\begin{lmm}\label{Dhb_lmm}
Let $X$ be a topological space and 
\BE{Dhb_e0a}\hb\!:\Hom(\De^k,X)\lra S_k(\De^k;R), 
\qquad k\!\in\!\Z^{\ge0},\EE
be a pre-chain map.
If $\hb$ vanishes on $\Hom(\De^0,X)$, then there exists a null-homotopy
$$ D_{\hb}\!:\Hom(\De^k,X)\lra S_{k+1}(\De^k;R), \qquad k\!\in\!\Z^{\ge0}\,,$$
for~$\hb$.
\end{lmm}

\begin{proof}
We take $D_{\hb}\!=\!0$ on $\Hom(\De^0,X)$.
Suppose $k\!\in\!\Z^+$ and we have constructed $D_{\hb}$ on $\Hom(\De^l,X)$
with $l\!<\!k$ so that it satisfies~\e_ref{hbrigid_e7} 
on $\Hom(\De^l,X)$ with $l\!<\!k$.
Let \hbox{$\si\!\in\!\Hom(\De^k,X)$} and
$$c_{\si}=\hb(\si)- \sum_{p=0}^k(-1)^p
\io_{k;p\,\#}\big(D_{\hb}(\si\!\circ\!\io_{k;p})\big).$$
For $k\!\ge\!2$,  the inductive assumption gives
\begin{equation*}\begin{split}
\prt_{\De^k}(c_{\si})
&=\prt_{\De^k}\big(\hb(\si)\!\big)-
\sum_{p=0}^k(-1)^p\io_{k;p\,\#}
\big(\prt_{\De^{k-1}}D_{\hb}(\si\!\circ\!\io_{k;p})\!\big)\\
&=\prt_{\De^k}\big(\hb(\si)\!\big)-
\sum_{p=0}^k(-1)^p\io_{k;p\,\#}\bigg(\!\!\hb(\si\!\circ\!\io_{k;p})
-\sum_{q=0}^{k-1}(-1)^q\io_{k-1;q\,\#}
D_{\hb}\big(\si\!\circ\!\io_{k;p}\!\circ\!\io_{k-1;q}\big)\!\!\bigg)\,.
\end{split}\end{equation*}
The terms in the double sum cancel in pairs, while the remaining difference
vanishes by~\e_ref{hbrigid_e4}.
For $k\!=\!1$,  \e_ref{hbrigid_e4} and the vanishing of~$\hb$ and~$D_{\hb}$ on~$\Hom(\De^0,X)$
imply~that 
$$\prt_{\De^k}c_{\si}=0$$
in this case as well. Since $H_k(\De^k;R)$ is trivial, there exists
$$D_{\hb}\si\in S_{k+1}(\De^k;R)
\qquad\hbox{s.t.}\quad 
\prt_{\De^k}\big(D_{\hb}(\si)\!\big)=c_{\si}\,.$$
This completes the inductive step.
\end{proof}

\noindent
A Hausdorff topological space $X'$ is \sf{locally compact} if
for every point $x\!\in\!X'$ there exists an open neighborhood~$U_x$ of~$x$ in~$X'$ 
such that the closure~$\ov{U_x}$ of~$U_x$ in~$X'$ is compact
(if $X'$ is not necessarily Hausdorff, there are various versions of this definition
that are equivalent for Hausdorff spaces).

\begin{lmm}\label{BMpush_lmm}
Let $f\!:X\!\lra\!X'$ be a proper map between topological spaces.
If either $X$ is compact or $X'$ is locally compact, 
then the~map 
\BE{BMpush_e3}\begin{split}  
&\hspace{1in} f_{\#}\!:\Scl_*(X;R)\lra\Scl_*(X';R), \\
&\big\{f_{\#}(c)\!\big\}(\tau)= \sum_{\si\in\Hom(\De^k,X)}\!\!\!\!\!\!\!\!\!\!\!\!
c(\si)\big\{\!f\!\circ\!\si\!\big\}(\tau)\in R
\qquad\forall~\tau\!\in\!\Hom(\De^p,X'),
\end{split}\EE
is a well-defined homomorphism of chain complexes  and
\BE{BMpush_e5} 
\supp\big(f_{\#}(c)\!\big)\subset f\big(\supp(c)\!\big)\qquad\forall~c\!\in\!\Scl_k(X;R).\EE
If $g\!:X'\!\lra\!X''$ is another proper continuous map and 
either $X'$ is compact or $X''$ is locally compact, then
\BE{BMpush_e6}\big(g\!\circ\!f\big)_{\#}\!=\!g_{\#}\!\circ\!f_{\#}\!:
\Scl_*(X;R)\lra\Scl_*(X'';R).\EE
\end{lmm}

\begin{proof} If $X$ is compact, the map~\e_ref{BMpush_e3} is the composition
$$\Scl_*(X;R)\!=\!S_*(X;R) \lra S_*(X';R)\lra\Scl_*(X';R).$$
The first arrow above is the pushforward homomorphism 
of the standard singular homology theory.\\

\noindent
For all $c\!\in\!\Scl_k(X;R)$ and $\tau\!\in\!\Hom(\De^p,X')$,
\begin{equation*}\begin{split}
\big\{\si\!\in\!\Hom(\De^k,X)\!:c(\si)\big\{f\!\circ\!\si\!\big\}(\tau)\!\neq\!0\big\}
=\big\{\si\!\in\!\Hom(\De^k,X)\!:c(\si)\!\neq\!0,\,
\tau\!=\!f\!\circ\!\si\big\}&\\
\subset \big\{\si\!\in\!\Hom(\De^k,X)\!:c(\si)\!\neq\!0,\,
\si(\De^k)\!\cap\!f^{-1}(\tau(\De^p))\!\neq\!\eset\big\}\subset 
\ale_c\big(f^{-1}(\tau(\De^p)\!)\!\big)&.
\end{split}\end{equation*}
If $f$ is a proper map, then  $f^{-1}(\tau(\De^p))$  is a compact subset of~$X$
and thus the last set above is finite.
This implies that the sum in~\e_ref{BMpush_e3} is finite.\\

\noindent
For all $c\!\in\!\Scl_k(X;R)$ and $U\!\subset\!X'$,
$$\ale_{f_{\#}(c)}(U) \subset
\big\{f\!\circ\!\si\!:\si\!\in\!\Hom(\De^k,X),\, c(\si)\!\neq\!0,\,
f\big(\si(\De^k)\big)\!\cap\!\ov{U}\neq\!\eset\big\}
=\big\{f\!\circ\!\si\!:\si\!\in\!\ale_c\big(f^{-1}(\ov{U})\!\big)\!\big\}.$$
If $f$ is a proper map and $\ov{U}\!\subset\!X'$ is compact, then $f^{-1}(\ov{U})$ is 
a compact subset of~$X$ and thus the last set above is finite.
This implies that $f_{\#}(c)\!\in\!\Scl_k(X;R)$ if in addition 
$X'$ is locally compact.\\

\noindent
It is immediate that the map~$f_{\#}$ in~\e_ref{BMpush_e3} is a homomorphism of $R$-modules
intertwining~$\prt_X$ and~$\prt_{X'}$ and that~\e_ref{BMpush_e6} holds.
Furthermore,
$$\big\{\tau\!\in\!\Hom(\De^p,X')\!:\big\{\!f_{\#}(c)\!\big\}(\tau)\!\neq\!0\big\}
\subset \big\{f\!\circ\!\si\!:\si\!\in\!\Hom(\De^k,X),\,c(\si)\!\neq\!0\big\}
\quad\forall\,c\!\in\!\Scl_k(X;R).$$
This establishes~\e_ref{BMpush_e5}.
\end{proof}

\noindent
In the notation of~\e_ref{BMch_e}, 
$$f_{\#}(c)=\sum_{\si\in\Hom(\De^k,X)}\!\!\!\!\!\!\!\!\!\!\!\!a_{\si}(f\!\circ\!\si)\,.$$
By the second paragraph in the proof of Lemma~\ref{BMpush_lmm},
each $p$-simplex~$\tau$ in~$X'$ appears only finitely many times in this sum.
Thus, the sum above can be reduced to a sum as in~\e_ref{BMch_e}.
By the third paragraph in the proof of Lemma~\ref{BMpush_lmm}, $f_{\#}(c)$ satisfies 
the required local finiteness condition.
The corollary below is an immediate consequence of Lemma~\ref{BMpush_lmm}.

\begin{crl}\label{BMpush_crl}
Let $f\!:X\!\lra\!X'$ be a proper map between topological spaces.
If either $X$ is compact or $X'$ is locally compact, then 
the composition of the $k$-simplicies to~$X$ with~$f$ induces a homomorphism
$$f_*\!: \Hcl_k(X;R)\lra\Hcl_k(X';R)\,.$$
If $g\!:X'\!\lra\!X''$ is another proper continuous map and 
either $X'$ is compact or $X''$ is locally compact, then
$$\big(g\!\circ\!f\big)_*\!=\!g_*\!\circ\!f_*\!:
\Hcl_*(X;R)\lra\Hcl_*(X'';R).$$
\end{crl}

\subsection{Subcomplexes and quotients}
\label{BMsubcom_subs}

\noindent
For a collection $\cA$ of subsets of a topological space~$X$, let
$$\Scl_{\cA;*}(X;R)\subset \Scl_*(X;R)$$
denote the subset of chains~$c$ such that 
$$\big\{\si\!\in\!\Hom(\De^k,X)\!:c(\si)\!\neq\!0\big\}
\subset \bigcup_{U\in\cA}\!\!\Hom(\De^k,U)
\quad\forall\,k\!\in\!\Z^{\ge0}\,.$$
This subset is a chain sub-complex of $\Scl_*(X;R)$.
We denote its homology by~$\Hcl_{\cA;*}(X;R)$.
Let 
$$\Scl_*\big(X,\cA;R\big)=\frac{\Scl_*(X;R)}{\Scl_{\cA;*}(X;R)}$$
be the quotient complex and $\Hcl_*\big(X,\cA;R\big)$
be its homology.
If $W\!\subset\!X$ contains every $U\!\in\!\cA$, then
$\Scl_{\cA;*}(X;R)$ is a sub-complex of~$\Scl_{\{W\};*}(X;R)$.
In such a case, let 
$$\Scl_{\{W\};*}\big(X,\cA;R\big)=
\frac{\Scl_{\{W\};*}(X;R)}{\Scl_{\cA;*}(X;R)}$$
be the quotient complex and $\Hcl_{\{W\};*}\big(X,\cA;R\big)$
be its homology.\\

\noindent
By definition, $\Scl_{\{X\};*}(X;R)\!=\!\Scl_*(X;R)$.
If $U\!\subset\!W\!\subset\!X$ and $\ov{W}\!\subset\!X$ is compact, then
\BE{SclWcmpt_e}\Scl_{\{W\};*}(X,\{U\};R)=S_*(W,U;R)\EE
is the standard relative simplicial complex for the pair~$(W,U)$.
If $\cA$ is a collection of subsets of~$X$ and
$\{W_U\!:U\!\in\cA\}$ is a locally finite collection of disjoint subsets of~$X$ with union~$W$
so that $U\!\subset\!W_U$ for every $U\!\in\!\cA$,
then
\BE{SclWdisj_e}
\Scl_{\{W\};*}\big(X,\cA;R\big)=\prod_{U\in\cA}\!\!
\Scl_{\{W_U\};*}\big(X,\{U\};R\big).\EE

\begin{lmm}\label{BMpart_lmm}
Let $X$ be a topological space and $\cA$ be a collection 
of subsets of~$X$ with union \hbox{$W\!\subset\!X$}.
If
$$W=\bigcup_{U\in\cA}\!\!\big(\Int\!_WU\big)\,,$$
there exists a pre-chain map as in~\e_ref{Dhb_e0a}
such that 
\begin{equation}\label{BMpart_e2}
\si_{\#}\big(\hb(\si)\!\big)\in\Scl_{\cA;*}(X;\R)~\forall\,
\si\!\in\!\Hom(\De^k,W),\quad
\hb(\si)\!=\!\id_{\De^k}~\forall\,\si\!\in\!\Hom(\De^k,U),\,U\!\in\!\cA.
\end{equation}
\end{lmm}

\begin{proof} This lemma is established in \cite[Appendix~I]{Vicks}
in different terminology.
For any topological space~$Y$, let 
$$\sd_{\!Y}\!:S_*(Y;R)\lra S_*(Y;R) \quad\hbox{and}\quad
D_Y\!:S_*(Y;R)\lra S_{*+1}(Y;R)$$
be the barycentric subdivision operator and a natural chain homotopy 
from~$\sd_{\!Y}$ to the identity on~$S_*(Y;R)$; see \cite[Section~31]{Mu2}.
In particular,
\begin{equation}\label{BMpart_e4}
\sd_{\!Y}\!-\!\id_{S_*(Y;R)}\!=\!\prt_YD_{\!Y}\!+\!D_{\!Y}\prt_Y\!:
S_*(Y;R)\lra S_*(Y;R).
\end{equation}
By \cite[Theorem~31.3]{Mu2}, 
$$m(\si)\equiv \min\!\big\{m\!\in\!\Z^{\ge0}\!:\sd_{\!X}^m\si\!\in\!\Scl_{\cA;*}(X;R)\!\big\}<\i
\qquad\forall\,\si\!\in\!\Hom(\De^k,W)\,.$$
In particular, $m(\si)\!=\!0$ if $\si\!\in\!\Scl_{\cA;*}(X;R)$ and
$m(\si\!\circ\!\io_{k;q})\!\le\!m(\si)$ for all $q\!=\!0,1,\ldots,k$.
Define~\e_ref{Dhb_e0a}~by
$$\hb(\si)=\sd_{\!\De^k}^{\!m(\si)}\id_{\De^k}-
D_{\De^k}\sum_{q=0}^k(-1)^q\!\!\!\!\!\!\!\!\!\!\sum_{r=m(\si\circ\io_{k;q})}^{m(\si)-1}
\hspace{-.22in}\sd_{\!\De^k}^r\io_{k;q}\in S_k(\De^k;R)\,.$$ 
By~\e_ref{BMpart_e4} and the naturality of~$\sd_Y$ and~$D_Y$, 
the collection of maps~$\hb$ with $k\!\in\!\Z^{\ge0}$ defined in this way is a pre-chain map.
By construction, this collection satisfies~\e_ref{BMpart_e2}.
\end{proof}

\begin{rmk}\label{BMpart_rmk}
The proof of Lemma~\ref{BMpart_lmm} defines a pre-chain map~$\hb$ as in~\e_ref{Dhb_e0a} 
only~on 
$$\Hom(\De^k,W)\subset\Hom(\De^k,X),$$ which suffices for our purposes below.
We can define~$\hb(\si)$ for~$\si$ in \hbox{$\Hom(\De^k,X)\!-\!\Hom(\De^k,W)$}
by taking $m(\si)\!=\!0$ 
if $\si$ does not map any of the simplicies of~$\De^k$ to~$W$
and the largest value of~$m(\si|_{\De'})$ taken over the simplicies $\De'\!\subset\!\De$
such that $\si(\De')\!\subset\!W$ if such a simplex~$\De'$ exists. 
\end{rmk}

\begin{crl}\label{BMpart_crl}
Let $X$ be a topological space and $\cA$ be a collection 
of subsets of~$X$ with union \hbox{$W\!\subset\!X$}. If
$$W=\bigcup_{U\in\cA}\!\!\big(\Int\!_WU\big)\,,$$
then the inclusion of $\Scl_{\cA;*}(X;R)$ into $\Scl_{\{W\};*}(X;R)$ is 
a chain homotopy equivalence.
If in addition \hbox{$W\!\subset\!Y\!\subset\!X$}, then the homomorphism
$$\Hcl_{\{Y\};*}(X,\cA;R)\lra \Hcl_{\{Y\};*}(X,\{W\};R)$$
induced by this inclusion is an isomorphism.
\end{crl}

\begin{proof} Let $\hb$ be the pre-chain map of Lemma~\ref{BMpart_lmm} 
(and Remark~\ref{BMpart_rmk}).
By Lemma~\ref{Dhb_lmm} applied to the pre-chain map
$$\Hom(\De^k,X)\lra S_k(\De^k;R),
\quad \si\mapsto \hb(\si)\!-\!\id_{\De^k}, \qquad k\!\in\!\Z^{\ge0},$$
the homomorphism 
$$\hb_{\#}\!:\Scl_{\{W\};*}(X;R)\lra \Scl_{\cA;*}(X;R)\!\subset\!
\Scl_{\{W\};*}(X;R)$$
induced by~$\hb$ is a chain homotopy inverse 
for the inclusion $\iota$ of $\Scl_{\cA;*}(X;R)$ into $\Scl_{\{W\};*}(X;R)$. That is, $D_{\hbar}$ restricts to the relevant subspaces by naturality (cf. \cite[Section 32]{Mu2}) and provides a null-homotopy of $\hbar_{\#}\circ \iota -\id$.\\

\noindent
The second claim follows from the commutativity of the diagram
$$\xymatrix{\ldots\ar[r]\ar[d]& \Hcl_{\cA;k}(X)\ar[r]\ar[d]^{\cong}&
\Hcl_{\{Y\};*}(X)\ar[r]\ar[d]^{\id}&
\Hcl_{\{Y\};k}(X,\cA)\ar[r]\ar[d]& \Hcl_{\cA;k-1}(X)\ar[r]\ar[d]^{\cong}\ar[r]\ar[d]
&\ldots\ar[d]^{\id}\\
\ldots\ar[r]&
\Hcl_{\{W\};k}(X)\ar[r]& \Hcl_{\{Y\};*}(X)\ar[r]&
\Hcl_{\{Y\};k}(X,\{W\})\ar[r]& \Hcl_{\{W\};k-1}(X)\ar[r]&\ldots}$$
where the rows are the long exact sequences for the pairs
$$\Scl_{\cA;*}(X;R)\subset\Scl_{\{Y\};*}(X;R) \quad\hbox{and}\quad 
\Scl_{\{W\};*}(X;R)\subset\Scl_{\{Y\};*}(X;R)$$
with the coefficient ring~$R$ omitted,
the first claim, and the Five Lemma.
\end{proof}

\noindent
For $U\!\subset\!W\!\subset\!X$, denote by 
$$\io_{W,U}\!:\Scl_{\{U\};*}(X;R)\lra\Scl_{\{W\};*}(X;R)
\quad\hbox{and}\quad
\io_{W,U*}\!:\Hcl_{\{U\};*}(X;R)\lra\Hcl_{\{W\};*}(X;R)$$
the inclusion homomorphism and the induced homomorphism on homology.
If in addition \hbox{$W\!\subset\!Y\!\subset\!X$}, denote~by
\begin{equation*}\begin{split}
j_{W,U}^Y\!:\Scl_{\{Y\};*}(X,\{U\};R)&\lra\Scl_{\{Y\};*}(X,\{W\};R)
\qquad\hbox{and}\\
j_{W,U*}^Y\!:\Hcl_{\{Y\};*}(X,\{U\};R)&\lra\Hcl_{\{Y\};*}(X,\{W\};R)
\end{split}\end{equation*}
the homomorphisms induced by the inclusion $U\!\subset\!W$ 
and the induced homomorphism on homology.

\begin{crl}[Mayer-Vietoris]\label{MVSpanier_crl}
Let $X$ be a topological space and $U,V\!\subset\!X$ be subsets such that 
$$U\!\cup\!V=\big(\Int\!_{U\cup V}\!U\big)\!\cup\!\big(\Int\!_{U\cup V}\!V\big).$$
Then there is a homomorphism
$$\prt\!:\Hcl_{\{U\cup V\};*}(X;R)\lra \Hcl_{\{U\cap V\};*-1}(X;R),$$
which is natural with respect to the homomorphisms induced by the admissible
inclusions \hbox{$U\!\subset\!U'$} and \hbox{$V\!\subset\!V'$}, so that 
the sequence
\begin{equation*}\begin{split}
\ldots\stackrel{\prt}{\lra} 
\Hcl_{\{U\cap V\};k}(X;R) &\xlra{(\io_{U,U\cap V*},\io_{V,U\cap V*})} 
\Hcl_{\{U\};k}(X;R)\!\oplus\!\Hcl_{\{V\};k}(X;R)\\
&\xlra{\io_{U\cup V,U*}-\io_{U\cup V,V}*} \Hcl_{\{U\cup V\};k}(X;R)
\stackrel{\prt}{\lra} \Hcl_{\{U\cap V\};k-1}(X;R) \lra\ldots
\end{split}\end{equation*}
of $R$-modules is exact.
\end{crl}

\begin{proof} For $A\!=\!U,V$, let
$\io_A\!:\Scl_{\{A\};*}(X;R)\!\lra\!\Scl_{\{U,V\};*}(X;R)$
denote the inclusion.
The short exact sequence
\begin{equation*}\begin{split}
0\lra
\Scl_{\{U\cap V\};*}(X;R) \xlra{(\io_{U,U\cap V},\io_{V;U\cap V})} 
\Scl_{\{U\};*}(X;R)\!\oplus\!\Scl_{\{V\};*}(X;R)&\\
\xlra{\io_U-\io_V} \Scl_{\{U,V\};k}(X;R)&\lra0
\end{split}\end{equation*}
of chain complexes is exact.
Thus, the claim follows from the Snake Lemma and the first claim of 
Corollary~\ref{BMpart_crl} with $\cA\!=\!\{U,V\}$.
\end{proof}

\begin{crl}[Relative Mayer-Vietoris]\label{relMVSpanier_crl}
Let $U,V\!\subset\!X$ be as in Corollary~\ref{MVSpanier_crl} and $W\!\subset\!X$
be such that $U\!\cup\!V\!\subset\!W$.
Then there is a homomorphism
$$\prt\!:\Hcl_{\{W\};*}\big(X,\{U\!\cup\!V\};R\big)\lra 
\Hcl_{\{W\};*-1}\big(X,\{U\!\cap\!V\};R\big),$$
which is natural with respect to the homomorphisms induced by the admissible
inclusions \hbox{$U\!\subset\!U'$}, \hbox{$V\!\subset\!V'$}, and \hbox{$W\!\subset\!W'$},
so that 
the sequence
\begin{equation*}\begin{split}
\ldots&\stackrel{\prt}{\lra} 
\Hcl_{\{W\};k}(X,\{U\!\cap\!V\};R) \xlra{(j^W_{U,U\cap V*},j^W_{V;U\cap V*})} 
\Hcl_{\{W\};k}(X,\{U\};R)\!\oplus\!\Hcl_{\{W\};k}(X,\{V\};R)\\
&\hspace{.7in}\xlra{j^W_{U\cup V,U*}-j^W_{U\cup V,V}*} \Hcl_{\{W\};k}(X,\{U\!\cup\!V\};R)
\stackrel{\prt}{\lra} \Hcl_{\{W\};k-1}(X,\{U\!\cap\!V\};R) \lra\ldots
\end{split}\end{equation*}
of $R$-modules is exact.
\end{crl}

\begin{proof} For $A\!=\!U,V$, let
$$j_A^W\!:\Scl_{\{W\};*}\big(X,\{A\};R\big)\lra
\Scl_{\{W\};*}\big(X,\{U,V\};R\big)$$
denote the homomorphism induced by the inclusion $\io_A$ in
the proof of Corollary~\ref{MVSpanier_crl}.
The short exact sequence
\begin{equation*}\begin{split}
0\lra \Scl_{\{W\};*}(X,\{U\!\cap\!V\};R) 
\xlra{(j^W_{U,U\cap V},j^W_{V;U\cap V})} 
\Scl_{\{W\};*}(X,\{U\};\Z)\!\oplus\!\Scl_{\{W\};*}(X,\{U\};R)&\\
\xlra{j_U^W-j_V^W} \Scl_{\{W\};*}(X,\{U,V\};R)&\lra0
\end{split}\end{equation*}
of chain complexes is then exact.
Thus, the claim follows from the Snake Lemma and 
the second claim of Corollary~\ref{BMpart_crl} with $\cA\!=\!\{U,V\}$ and $Y\!=\!W$.
\end{proof}

\begin{crl}[Excision]\label{SpBMexcise_crl}
Let $X$ be a topological space and $U,W\!\subset\!X$ be subspaces
such that the closure of $X\!-\!U$ in~$X$ is contained in~$\Int\,W$.
Then the homomorphism
\BE{SpBMexcise_e0}\io_*\!:  \Hcl_{\{W\};*}(X,\{U\!\cap\!W\};R)\lra \Hcl_*(X,\{U\};R)\EE
induced by the inclusion $(W,U\!\cap\!W)\!\lra\!(X,U)$ is an isomorphism.
\end{crl}

\begin{proof} Let $\cA\!=\!\{U,W\}$.
The homomorphism~\e_ref{SpBMexcise_e0} is induced by the composition
\BE{SpBMexcise_e3}
\frac{\Scl_{\{W\};*}(X;R)}{\Scl_{\{U\cap W\};*}(X;R)}
\lra\frac{\Scl_{\cA;*}(X;R)}{\Scl_{\{U\};*}(X;R)}
\lra\frac{\Scl_*(X;R)}{\Scl_{\{U\};*}(X;R)}\EE
of homomorphisms of chain complexes.
The first homomorphism above is an isomorphism.
By the assumptions, the interiors of~$U$ and~$W$ cover~$X$.
By the first claim of Corollary~\ref{BMpart_crl} and the Five Lemma, 
the second homomorphism in~\e_ref{SpBMexcise_e3} thus also induces
an isomorphism in homology.
\end{proof}

\subsection{Fundamental class}
\label{FundClass_subs}

\noindent
For a topological space $X$, subsets $A\!\subset\!B\!\subset\!X$ and $W\!\subset\!X$, and 
a class $\mu\!\in\!\Hcl_{\{W\};*}(X,\{W\!-\!B\};R)$, we denote~by
$$\mu|_A\in \Hcl_{\{W\};*}\big(X,\{W\!-\!A\};R\big)$$
the image of $\mu$ under the homomorphism
\BE{Hclrestr_e}\Hcl_{\{W\};*}(X,\{W\!-\!B\};R)\lra \Hcl_{\{W\};*}(X,\{W\!-\!A\};R)\EE
induced by the inclusion $(W,W\!-\!B)\!\lra\!(W,W\!-\!A)$.\\

\noindent
Let $X$ be an $n$-manifold and $B\!\subset\!X$ be a ball (open or closed) around a point $x\!\in\!X$.  
By Corollary~\ref{SpBMexcise_crl} with $W\!=\!B$, \e_ref{SclWcmpt_e}, and the Kunneth formula,
$$\Hcl_k\big(X,\{X\!-\!\{x\}\!\};R\big)\cong 
\Hcl_{\{B\};k}\big(X,\{B\!-\!\{x\}\!\};R\big)
= H_k\big(B,B\!-\!\{x\};R\big)
\cong \begin{cases}R,&\hbox{if}~k\!=\!n;\\
\{0\},&\hbox{otherwise}.
\end{cases}$$
An \sf{$R$-orientation for~$X$ at $x\!\in\!X$} is a choice of generator 
\hbox{$\mu_x\!\in\!\Hcl_n(X,X\!-\!\{x\};R)$}.
An \sf{$R$-orientation for~$X$} is a collection $(\mu_x)_{x\in X}$
of $R$-orientations for~$X$ at~$x$ so that for every $x\!\in\!X$ there exist a neighborhood $U\!\subset\!X$
of~$x$ and $\mu_U\!\in\!\Hcl_k(X,\{X\!-\!U\};R)$ such that
$$\mu_U\big|_y=\mu_y\in\Hcl_n\big(X,\{X\!-\!\{y\}\!\};R\big) \qquad\forall\,y\!\in\!U.$$
An \sf{$R$-oriented manifold} is a pair~$(X,(\mu_x)_{x\in X})$ consisting of a manifold~$X$
and an orientation~$(\mu_x)_{x\in X}$ for~$X$.
By Proposition~\ref{FC_prp}\ref{mnfdFC_it} below with $A\!=\!X$, 
an $R$-oriented $n$-manifold $(X,(\mu_x)_{x\in X})$ carries a \sf{fundamental class}
$$[X]\!\equiv\!\mu_X\!\in\!\Hcl_n(X,\eset;R)\!\equiv\!\Hcl_n(X;R).$$

\begin{prp}[Fundamental Class]\label{FC_prp}
Let $X$ be an $n$-manifold and $A\!\subset\!X$ be a closed subset.
\begin{enumerate}[label=(\arabic*),leftmargin=*]

\item\label{mnfdvan_it} For every $k\!>\!n$, $\Hcl_k(X,\{X\!-\!A\};R)\!=\!0$.

\item\label{mnfinj_it} An element $\mu_A\!\in\!\Hcl_n(X,\{X\!-\!A\};R)$ is zero
if and only if
$$\mu_A\big|_x=0\in\Hcl_n\big(X,\{X\!-\!\{x\}\!\};R\big)  \qquad\forall\,x\!\in\!A.$$

\item\label{mnfdFC_it} If $(\mu_x)_{x\in X}$ is an $R$-orientation on~$X$,
there exists a unique $\mu_A\!\in\!\Hcl_n(X,\{X\!-\!A\};R)$ such~that
\BE{mnfdFC_e}\mu_A\big|_x=\mu_x\in\Hcl_n\big(X,\{X\!-\!\{x\}\!\};R\big) \qquad\forall\,x\!\in\!A.\EE

\end{enumerate}
\end{prp}

\begin{proof}[{\bf{\emph{Proof of Proposition~\ref{FC_prp}\ref{mnfdvan_it},\ref{mnfinj_it}}}}]
The proof is divided into four steps.\\

\noindent
{\it Case 1.} Suppose $A$ is compact. 
Let $U\!\subset\!X$ be a precompact open neighborhood of~$A$. 
By Corollary~\ref{SpBMexcise_crl} with $W\!=\!U$ and~\e_ref{SclWcmpt_e},
\BE{FCcmpt_e1}\Hcl_*\big(X,\{X\!-\!A\!\};R\big)\cong 
\Hcl_{\{U\};*}\big(X,\{U\!-\!A\!\};R\big)
= H_*\big(U,U\!-\!A;R\big).\EE
The two claims in this case thus follow from \cite[Lemma~A.7]{MiSt}.\\

\noindent
{\it Case 2.} Suppose $A$ is the union of a locally finite collection~$\cA$ 
of disjoint compact subsets of~$X$.
Let $\{U_B\!:B\!\in\!\cA\}$ be a locally finite collection of disjoint precompact open subsets
of~$X$ so that $B\!\subset\!U_B$ for every $B\!\in\!\cA$.
Let $U\!\subset\!X$ be the union of the subsets~$U_B$.
By Corollary~\ref{SpBMexcise_crl} with $W\!=\!U$ and~\e_ref{SclWdisj_e},
\BE{FCcount_e1}\begin{split}
\Hcl_*\big(X,\{X\!-\!A\!\};R\big)
&\cong \Hcl_{\{U\};*}\big(X,\{U\!-\!A\!\};R\big)\\
&\cong \prod_{B\in\cA}\!\!\Hcl_{\{U_B\};*}\big(X,\{U_B\!-\!B\};R\big)
\cong \prod_{B\in\cA}\!\!\Hcl_*\big(X,\{X\!-\!B\};R\big).
\end{split}\EE
The composition of the above isomorphism with the projection to the $B$-th component
of the product is the restriction homomorphism
\BE{FCcount_e3}\Hcl_*\big(X,\{X\!-\!A\!\};R\big)\lra\Hcl_*\big(X,\{X\!-\!B\};R\big).\EE
The two claims in this case thus follow from Case~1.\\

\noindent
{\it Case 3.} Suppose $A_1,A_2\!\subset\!X$ are closed, $A\!=\!A_1\!\cup\!A_2$,
and the two claims hold for the subsets \hbox{$A_1,A_2,A_1\!\cap\!A_2$} of~$X$.
By Corollary~\ref{relMVSpanier_crl} with $W\!=\!X$, $U\!=\!X\!-\!A_1$, and $V\!=\!X\!-\!A_2$,
there is an exact sequence
\BE{FCcup_e1}\begin{split}
\ldots\lra\Hcl_{k+1}\big(X,\{X\!-\!A_1\!\cap\!A_2\};R\big)
&\lra \Hcl_k\big(X,\{X\!-\!A\};R\big) \\
&\lra \Hcl_k(X,\{X\!-\!A_1\};R)\!\oplus\!\Hcl_k\big(X,\{X\!-\!A_2\};R\big)\lra\ldots
\end{split}\EE
Thus, the two claims also hold for~$A$.\\

\noindent
{\it Case 4.} $A$ is arbitrary.
Let $\{A_i\}_{i\in\Z}$ be a locally finite collection of compact subsets of~$X$ such~that 
$$A=\bigcup_{i\in\Z}A_i \qquad\hbox{and}\qquad A_i\!\cap\!A_j= \eset~~\hbox{if}~|i\!-\!j|\!>\!1.$$ 
By Case~2, the two claims hold for the~subsets
$$A_{odd}\equiv\bigcup_{i\in\Z}A_{2i-1}, \qquad
A_{even}\equiv\bigcup_{i\in\Z}A_{2i}, \quad\hbox{and}\quad
A_{odd}\!\cap\!A_{even}=\bigcup_{i\in\Z}A_i\!\cap\!A_{i+1}$$
of~$X$.
By Case~3, the two claims hold for $A\!\equiv\!A_{odd}\!\cup\!A_{even}$ as well.
\end{proof}

\begin{proof}[{\bf{\emph{Proof of Proposition~\ref{FC_prp}\ref{mnfdFC_it}}}}]
The uniqueness of~$\mu_A$ follows immediately from the second claim of the proposition.
The uniqueness property implies~that 
\begin{equation}\label{mnfdFC_e3}
\mu_{A'}=\mu_A|_{A'}\in \Hcl_n(X,\{X\!-\!A'\};R)
\end{equation}
whenever $A'\!\subset\!A$ and an element $\mu_A\!\in\!\Hcl_n(X,\{X\!-\!A\};R)$
satisfying~\e_ref{mnfdFC_e} exists.
The existence proof is again divided into four steps.\\

\noindent
{\it Case 1.} Suppose $A$ is compact. 
Let $U\!\subset\!X$ be a precompact open neighborhood of~$A$. 
The claim in this case follows from \e_ref{FCcmpt_e1} with $*\!=\!n$ and 
\cite[Theorem~A.8]{MiSt}.\\

\noindent
{\it Case 2.} Suppose $A$ is the union of a locally finite collection~$\cA$ 
of disjoint compact subsets of~$X$.
Let $\{U_B\!:B\!\in\!\cA\}$ and $U\!\subset\!X$ be as in Case~2 
in the proof of Proposition~\ref{FC_prp}\ref{mnfdvan_it},\ref{mnfinj_it}.
Since the composition of the isomorphism~\e_ref{FCcount_e1} 
with the projection to the $B$-th component
of the product is the restriction homomorphism~\e_ref{FCcount_e3},
the preimage~$\mu_A$ of the element~$(\mu_B)_{B\in\cA}$ under this isomorphism
satisfies~\e_ref{mnfdFC_e}.\\

\noindent
{\it Case 3.} Suppose $A_1,A_2\!\subset\!X$ are closed, $A\!=\!A_1\!\cup\!A_2$,
and the claim holds for the subsets \hbox{$A_1,A_2,A_1\!\cap\!A_2$} of~$X$.
By the first claim of the proposition,
the long exact sequence~\e_ref{FCcup_e1} becomes
\begin{equation*}\begin{split}
0\lra \Hcl_n\big(X,\{X\!-\!A\};R\big) 
&\lra \Hcl_n(X,\{X\!-\!A_1\};R)\!\oplus\!\Hcl_n\big(X,\{X\!-\!A_2\};R\big)\\
&\lra \Hcl_n(X,\{X\!-\!A_1\!\cap\!A_2\};R)\lra\ldots
\end{split}\end{equation*}
By~\e_ref{mnfdFC_e3}, $\mu_{A_1}|_{A_1\cap A_2}\!=\!\mu_{A_1\cap A_2}\!=\!\mu_{A_2}|_{A_1\cap A_2}$.
Thus, there exists
$$\mu_A\in\Hcl_n\big(X,\{X\!-\!A\};R\big) \qquad\hbox{s.t.}\quad
\mu_A|_{A_1}\!=\!\mu_{A_1},~~\mu_A|_{A_2}\!=\!\mu_{A_2}.$$
Since $\mu_A|_x\!=\!\mu_{A_i}|_x$ for all $x\!\in\!A_i$,
$\mu_A$ satisfies~\e_ref{mnfdFC_e}.\\

\noindent
{\it Case 4.} $A$ is arbitrary.
Let $\{A_i\}_{i\in\Z}$ be as in Case~4
in the proof of Proposition~\ref{FC_prp}\ref{mnfdvan_it},\ref{mnfinj_it}.
By Case~2, the claim holds for the~subsets
$$A_{odd}\equiv\bigcup_{i\in\Z}A_{2i-1}  \qquad\hbox{and}\qquad
A_{even}\equiv\bigcup_{i\in\Z}A_{2i}\,.$$
By Case~3, the claims holds for $A\!\equiv\!A_{odd}\!\cup\!A_{even}$ as well.
\end{proof}

\subsection{Poincar\'e Duality}
\label{PD_subs}

\noindent
For a collection $\cA$ of subsets of a topological space~$X$
and a subset $W\!\subset\!X$ containing every \hbox{$U\!\in\!\cA$},
the homomorphism~\e_ref{Scapdfn_e} induces a homomorphism
$$\cap\!: S^q(W;R)\!\otimes_R\!\Scl_{\{W\};p+q}(X,\cA;R)\lra \Scl_{\{W\};p}(X,\cA;R).$$
The latter in turn induces a natural homomorphism
\BE{caprestr_e}\cap\!: H^q(W;R)\!\otimes_R\!\Hcl_{\{W\};p+q}(X,\cA;R)\lra \Hcl_{\{W\};p}(X,\cA;R).\EE
For $U,W'\!\subset\!W$, let
$$\big\{\io_{W,W'}\big\}_{\!*}\!:
\Hcl_{\{W'\};p}\big(X,\{U\!\cap\!W'\};R\big)\lra\Hcl_{\{W\};p}\big(X,\{U\};R\big)$$
be the homomorphism induced by the inclusion $(W',U\!\cap\!W')\!\lra\!(W,U)$.
By the naturality of~\e_ref{caprestr_e},
\BE{CapNat_e}\begin{split} 
&\big\{\io_{W,W'}\big\}_{\!*}\big(\!(\al|_{W'})\!\cap\!\mu\big)=
\al\!\cap\!\big(\{\io_{W,W'}\}_*(\mu)\!\big)\in 
\Hcl_{\{W\};p}\big(X,\{U\};R\big)\\
&\hspace{2in}
\forall~\al\!\in\!H^q(W;R),\,
\mu\!\in\!\Hcl_{\{W'\};p+q}\big(X,\{U\!\cap\!W'\};R\big).
\end{split}\EE

\vspace{.2in}

\noindent
For subsets $U,W$ of a topological space~$X$ such that 
the closure of $X\!-\!U$ in~$X$ is contained in~$\Int\,W$ and 
\hbox{$\mu\!\in\!\Hcl_*(X,\{U\};R)$}, we denote~by
$$\mu|_W\in \Hcl_{\{W\};*}\big(X,\{U\!\cap\!W\};R\big)$$
the preimage of $\mu$ under the excision isomorphism~\e_ref{SpBMexcise_e0}.
If $W'\!\subset\!W$ is another subset such that 
the closure of $X\!-\!U$ in~$X$ is contained in~$\Int\,W'$, then
\BE{SpBMexcise_e0a}\begin{split}
\big\{\io_{X,W'}\big\}_*\!=\!
\big\{\io_{X,W}\big\}_*\!\circ\!\big\{\io_{W,W'}\big\}_*\!:
\Hcl_{\{W'\};*}\big(X,\{U\!\cap\!W'\};R\big)
&\lra\Hcl_{\{W\};*}\big(X,\{U\!\cap\!W\};R\big)\\
&\lra\Hcl_*\big(X,\{U\};R\big)
\end{split}\EE
and thus
\BE{SpBMexcise_e0b} \big\{\io_{W,W'}\big\}_*\big(\mu|_{W'}\big)=\mu|_W\in 
\Hcl_{\{W\};*}\big(X,\{U\!\cap\!W\};R\big)
\quad\forall\,\mu\!\in\!\Hcl_*\big(X,\{U\};R\big).\EE

\vspace{.2in}

\noindent
Let $(X,(\mu_x)_{x\in X})$ be an $R$-oriented $n$-manifold, $A\!\subset\!X$ a closed subset,
and 
$$\mu_A\in\Hcl_n\big(X,\{X\!-\!A\};R\big)$$ 
the fundamental class provided by 
Proposition~\ref{FC_prp}\ref{mnfdFC_it}.
Suppose $U_A\!\subset\!X$ is an open neighborhood of~$A$ that deformation retracts onto~$A$.
Thus, the restriction homomorphism
$$H^*(U_A;R)\lra\ H^*(A;R), \qquad \al\mapsto\al|_{A}\,,$$
is an isomorphism.
It follows that the homomorphism
\BE{PDdfn_e}\begin{split}
&\hspace{.2in}\PD_{A;U_A}\!: H^k(A;R)\lra\Hcl_{n-k}\big(X,\{X\!-\!A\};R\big),\\
&\PD_{A;U_A}\big(\al|_A\big)=  
\big\{\io_{X,U_A}\big\}_*\big(\al\!\cap\!(\mu_A|_{U_A})\!\big)
~~\forall\,\al\!\in\!H^k(U_A;R),
\end{split}\EE
is well-defined.\\

\noindent
If $B\!\supset\!A$ is another closed subset of~$X$ and 
$U_B\!\subset\!X$ is an open neighborhood of~$B$ that deformation retracts onto~$B$ 
and contains~$U_A$,
\begin{gather}
\mu_A=\mu_B|_A\in \Hcl_n\big(X,\{X\!-\!A\};R\big),\notag \\
\label{FCrestrcond_e}
\big\{\io_{U_B,U_A}\big\}_*\big(\mu_A|_{U_A}\big)
=\mu_A|_{U_B}
=\big(\mu_B|_{U_B}\big)\!\big|_A\in\Hcl_{\{U_B\};n}\big(X,\{U_B\!-\!A\};R\big)
\end{gather}
by the uniqueness part of Proposition~\ref{FC_prp}\ref{mnfdFC_it},
\e_ref{SpBMexcise_e0b}, and the commutativity of the diagram
\BE{restrCD_e}\begin{split}
\xymatrix{\Hcl_{\{U_B\};*}\big(X,\{U_B\!-\!B\};R\big)\ar[rr]^{\cdot|_A}
\ar[d]_{\cong}&&
\Hcl_{\{U_B\};*}\big(X,\{U_B\!-\!A\};R\big) \ar[d]^{\cong}\\
\Hcl_*\big(X,\{X\!-\!B\};R\big)\ar[rr]^{\cdot|_A}&&
\Hcl_*\big(X,\{X\!-\!A\};R\big)\,.}
\end{split}\EE
Along with~\e_ref{CapNat_e}, \e_ref{FCrestrcond_e} gives 
\begin{equation*}\begin{split}
\big\{\io_{U_B,U_A}\big\}_*\big(\al|_{U_A}\!\cap\!(\mu_A|_{U_A})\!\big)
&=\al\!\cap\!(\mu_A|_{U_B})=\al\!\cap\!\big(\!(\mu_B|_{U_B})|_A\big)\\
&=\big(\al\!\cap\!(\mu_B|_{U_B})\!\big)\!\big|_A\in 
\Hcl_{\{U_B\};n-k}\big(X,\{U_B\!-\!A\};R\big)
\quad\forall\,\al\!\in\!H^k(U_B;R).
\end{split}\end{equation*}
Combining this with~\e_ref{SpBMexcise_e0a} and the commutativity of~\e_ref{restrCD_e}, 
we conclude that the diagram
\BE{PDrestrCD_e}\begin{split}
\xymatrix{H^k(B;R) \ar[d]_{\PD_{B;U_B}}\ar[rr]^{\cdot|_A}&&  
H^k(A;R)\ar[d]^{\PD_{A;U_A}}\\
\Hcl_{n-k}\big(X,\{X\!-\!B\};R\big)\ar[rr]^{\cdot|_A}&& 
\Hcl_{n-k}\big(X,\{X\!-\!A\};R\big)}
\end{split}\EE
commutes.\\

\noindent
By the commutativity of~\e_ref{PDrestrCD_e} with $A\!=\!B$, 
the homomorphism~\e_ref{PDdfn_e} does not depend on
the choice of~$U_A$ if $A$ is a \sf{neighborhood retract},
i.e.~every open neighborhood~$W\!\subset\!X$ of~$A$ contains
an open neighborhood~$U_A$ of~$A$ that deformation retracts onto~$A$.
This is in particular the case if $A\!\subset\!X$ is a closed submanifold with corners.
If $A\!\subset\!X$ is a closed neighborhood retract, we denote~by
\BE{PDdfn_e2}\PD_A\!:H^k(A;R)\lra\Hcl_{n-k}\big(X,\{X\!-\!A\};R\big)\EE
the homomorphism~\e_ref{PDdfn_e} for any admissible neighborhood $U_A$ of~$A$.
For $A\!=\!X$, this homomorphism is given~by
$$\PD_X\!:H^k(X;R)\lra\Hcl_{n-k}\big(X;R\big), \qquad
\PD_X(\al)=\al\!\cap\![X].$$
If $A\!\subset\!B\!\subset\!X$ are closed neighborhood retracts, 
the commutativity of~\e_ref{PDrestrCD_e} implies that 
 the diagram
\BE{PDrestrCD_e2}\begin{split}
\xymatrix{H^k(B;R) \ar[d]_{\PD_B}\ar[rr]^{\cdot|_A}&&  
H^k(A;R)\ar[d]^{\PD_A}\\
\Hcl_{n-k}\big(X,\{X\!-\!B\};R\big)\ar[rr]^{\cdot|_A}&& 
\Hcl_{n-k}\big(X,\{X\!-\!A\};R\big)}
\end{split}\EE
commutes as well.

\begin{prp}[Poincar\'e Duality]\label{PD_prp}
Let $(X,(\mu_x)_{x\in X})$ be an $R$-oriented $n$-manifold. 
If $A\!\subset\!X$ is a closed $n$-submanifold with corners,
the homomorphism~\e_ref{PDdfn_e2} is an isomorphism.
\end{prp}

\begin{proof} The proof is again divided into four steps.\\

\noindent
{\it Case 1.} Suppose $A$ is compact.
Let $U\!\subset\!X$ be a precompact open neighborhood of~$A$ that deformation retracts onto~$A$. 
Combining the isomorphism~\e_ref{FCcmpt_e1} with 
the homotopy invariance of the standard singular homology
for $(U,U\!-\!A)\!\cong\!(A,\prt A)$, we obtain
\BE{PDcmpt_e1}\Hcl_*\big(X,\{X\!-\!A\!\};R\big)\cong  H_*\big(A,\prt A;R\big).\EE
Since~$\mu_A$ corresponds to the standard fundamental class 
\hbox{$[A,\prt A]\!\in\!H_n(A,\prt A;R)$} under this isomorphism,
the diagram 
$$\xymatrix{ H^k(A;R) \ar[d]_{\PD_{A;U}}\ar[rr]^{\id}&& H^k(A;R)\ar[d]^{\PD_{(A,\prt A)}}
&\al\ar[d]^{\PD_{(A,\prt A)}}  \\
\Hcl_*\big(X,\{X\!-\!A\};R\big) \ar[rr]_{\cong}^{\e_ref{PDcmpt_e1}}&& H_*(A,\prt A;R)
&\al\!\cap\![A,\prt A]}$$
commutes.
Since $(A,\prt A)$ is a compact topological manifold with boundary,
$\PD_{(A,\prt A)}$ is an isomorphism by the compact case of \cite[Exercise~A.1]{MiSt}
and the $(M,A,B)\!=\!(A,\eset,\prt A)$ case of \cite[Theorem~3.43]{Hatcher}.
Thus, $\PD_{A;U}$ is an isomorphism as~well.\\

\noindent
{\it Case 2.} Suppose $A$ is the union of a locally finite collection~$\cA$ 
of disjoint compact subsets of~$X$ so that each $B\!\in\!\cA$ is an $n$-submanifold
with corners.
Let $\{U_B\!:B\!\in\!\cA\}$ and $U\!\subset\!X$ be as in Case~2 
in the proof of Proposition~\ref{FC_prp}\ref{mnfdvan_it},\ref{mnfinj_it}
so that each~$U_B$ deformation retracts onto~$B$.
In particular, the restriction homomorphisms
$$H^*(A;R)\lra H^*(B;R) \quad\hbox{and}\quad H^*(U;R)\lra H^*(U_B;R)$$
induce isomorphisms
\BE{PDcount_e1}H^*(A;R)\cong \prod_{B\in\cA}\!\!H^*(B;R)
\quad\hbox{and}\quad
H^*(U;R)\cong \prod_{B\in\cA}\!\!H^*(U_B;R),\EE
respectively.
Since~$\mu_A$ corresponds to~$(\mu_B)_{B\in\cA}$ under the isomorphism~\e_ref{FCcount_e1},
the diagram 
$$\xymatrix{ H^*(A;R) \ar[d]_{\PD_{A;U}}\ar[rr]_{\cong}^{\e_ref{PDcount_e1}}&&
\prod\limits_{B\in\cA}\!\!\!H^*(B;R)\ar[d]^{\prod\limits_{B\in\cA}\!\!\!\PD_{B;U_B}}\\
\Hcl_*\big(X,\{X\!-\!A\!\};R\big)\ar[rr]_{\cong}^{\e_ref{FCcount_e1}}&&
\prod\limits_{B\in\cA}\!\!\!\Hcl_*\big(X,\{X\!-\!B\};R\big)}$$
commutes.
Thus, $\PD_{A;U}$ is an isomorphism by Case~1.\\

\noindent
{\it Case 3.} Suppose $A_1,A_2,A_1\!\cap\!A_2\!\subset\!X$ are closed $n$-submanifolds 
with corners which satisfy the claim and \hbox{$A\!=\!A_1\!\cup\!A_2$}. 
For a subspace $B\!\subset\!X$, let 
$$\sH_*(B)=\Hcl_*\big(X,\{X\!-B\};R\big) \qquad\hbox{and}\qquad
\sH^*(B)=H^*(B;R).$$
Let $A_{12}\!=\!A_1\!\cap\!A_2$.
For $i\!=\!1,2$, define
\begin{alignat*}{2}
\io_i\!:\sH_*(A)&\lra\sH_*(A_i), &\qquad j_i\!:\sH_*(A_i)&\lra\sH_*(A_{12}),\\
\io_i^*\!:\sH^*(A)&\lra\sH^*(A_i), &\qquad j_i^*\!:\sH^*(A_i)&\lra\sH^*(A_{12})
\end{alignat*}
to be the homology homomorphisms as in~\e_ref{Hclrestr_e} and
the usual cohomology restriction homomorphisms.
By Mayer-Vietoris for the standard singular cohomology and
Corollary~\ref{relMVSpanier_crl} with $W\!=\!X$, $U\!=\!X\!-\!A_1$, and $V\!=\!X\!-\!A_2$, 
the rows in the diagram
$$\xymatrix{\ldots\ar[r]& \sH^{k-1}(A_{12})\ar[d]_{\PD_{A_{12}}}
\ar[r]^>>>>>>>>{\de}&
\sH^k(A)\ar[d]_{\PD_A}\ar[r]^>>>>>>>{(\io_1^*,\io_2^*)}&
\sH^k(A_1)\!\oplus\!\sH^k(A_2)\ar[d]|{\PD_{A_1}\!\oplus\PD_{A_2}}
\ar[r]^<<<<<<<{j_1^*-j_2^*}&
\sH^k(A_{12})\ar[d]_{\PD_{A_{12}}}\ar[r]^<<<<<{\de}&\ldots\\
\ldots\ar[r]&
\sH_{n-k+1}(A_{12})\ar[r]^>>>>>>{\prt}& \sH_{n-k}(A)\ar[r]^>>>>>>{(\io_1,\io_2)}&
\sH_{n-k}(A_1)\!\oplus\!\sH_{n-k}(A_2)\ar[r]^<<<<{j_1-j_2}&
\sH_{n-k}(A_{12})\ar[r]^<<<<{\prt}&\ldots}$$
are exact. 
The second and third squares above commute by the commutativity of~\e_ref{PDrestrCD_e2}.
By~\e_ref{FCrestrcond_e} and~\e_ref{capbd_e},
the first square commutes up to the multiplication by~$(-1)^{n-k+1}$.
Since the homomorphisms~$\PD_{A_{12}},\PD_{A_1},\PD_{A_2}$ are isomorphisms, 
the Five Lemma implies that so are the homomorphisms~$\PD_A$.\\

\noindent
{\it Case 4.} $A$ is arbitrary.
Let $\{A_i\}_{i\in\Z}$ be as in Case~4
in the proof of Proposition~\ref{FC_prp}\ref{mnfdvan_it},\ref{mnfinj_it}
so that all $A_i,A_i\!\cap\!A_j\!\subset\!X$ are compact $n$-submanifolds with corners.
By Case~2, the claim holds for the~subsets
$$A_{odd}\equiv\bigcup_{i\in\Z}A_{2i-1}, \qquad
A_{even}\equiv\bigcup_{i\in\Z}A_{2i}, \quad\hbox{and}\quad
A_{odd}\!\cap\!A_{even}=\bigcup_{i\in\Z}A_i\!\cap\!A_{i+1}$$
of~$X$.
By Case~3, the claims holds for $A\!\equiv\!A_{odd}\!\cup\!A_{even}$ as well.
\end{proof}

\begin{rmk}\label{PD_rmk}
Let $A\!\subset\!X$ be a closed submanifold with corners and $f\!:X\!\lra\!\R$ 
a proper smooth function.
Choose a collection $(a_i)_{i\in\Z}$ of regular values of~$f$ and its restrictions to the strata of~$A$
so~that
$$a_i<a_j~~\forall\,i\!<\!j,\qquad \lim_{i\lra-\i}\!\!\!a_i=-\i, \qquad \lim_{i\lra\i}\!\!a_i=\i.$$ 
A decomposition as in Case~4 can then be obtained by taking
$A_i\!=\!f^{-1}([a_{2i-1},a_{2i+2}])$.
\end{rmk}

\section{Proof of Theorem~\ref{main_thm}}
\label{mainpf_sec}

\subsection{Homology of neighborhoods of smooth maps}
\label{neigh_subs}

\noindent 
The next proposition is an analogue of Proposition~2.2 in~\cite{Z} for
the Borel-Moore homology groups used in this paper.

\begin{prp}\label{neighb2_prp}
Let $h\!: Y\!\lra\!X$ be a smooth map between manifolds, 
$A\!\subset\!X$ be a closed subset so that $A\!\subset\!h(Y)$,
and $W\!\subset\!X$ be an open neighborhood of~$A$.
There exists an open neighborhood $U\!\subset\!W$ of $A$ such~that
$$\Hcl_{\{U\};l}(X;R)=0 \qquad\hbox{if}~~l\!>\!\dim Y.$$
\end{prp}

\vspace{.2in}

\noindent
If $h\!: Y\!\lra\!X$ is a smooth map and
$k$ is a nonnegative integer, put
$$N_k(h)=\big\{y\!\in\!Y\!: \rk \nd_yh\!\le\!k\big\}.$$
Proposition~\ref{neighb2_prp} follows from Lemma~\ref{neighb2_lmm}
applied with $k\!=\!\dim Y$.\\

\noindent
For a simplicial complex $K$, we denote by $|K|$ a geometric realization of~$K$ 
in a Euclidean space in the sense of \cite[Section~3]{Mu2}
and by $\sd K$ the barycentric subdivision of~$K$.
The simplicies of $\sd K$ are the~sets 
$$\tau=b_{\si_1}\ldots b_{\si_j}\equiv \big\{b_{\si_1},\ldots,b_{\si_j}\big\}
\qquad\hbox{with}\quad \si_1,\ldots,\si_j\!\in\!K, \quad
\si_1\!\subsetneq\!\ldots\!\subsetneq\!\si_j.$$
In a geometric realization $|K|\!=\!|\sd K|$, $b_{\si}$ corresponds to the barycenter 
of the simplex $\si\!\in\!K$.
For $l\!\in\!\Z^{\ge0}$, denote by $K_l\!\subset\!K$ the \sf{$l$-skeleton} of~$K$.\\

\noindent
For a simplex $\si\!\in\!K$, let
$$\St(\si,K)=
\bigcup_{\begin{subarray}{c}\si'\in K\\ \si\subset\si'\end{subarray}}\!\!\Int\si'
\subset |K|$$
be the (\sf{open}) \sf{star of $\si$ in~$K$};  see \cite[Section~62]{Mu2}.
Its closure in~$|K|$ is the \sf{closed~star}
$$\ov\St(\si,K)\equiv
\bigcup_{\begin{subarray}{c}\si'\in K\\ \si\subset\si'\end{subarray}}\!\!|\si'|
\subset |K|$$
of $\si$ in~$K$ and is compact if $\si$ is contained in only finitely many simplices.
\\

\noindent
A \sf{triangulation} of a manifold~$X$ is a pair
$T\!=\!(K,\eta)$ consisting of a simplicial complex
and a homeomorphism $\eta\!:|K|\!\lra\!X$
such that $\eta|_{\Int\si}$ is smooth for every \hbox{simplex $\si\!\in\!K$}.

\begin{lmm}\label{neighb2_lmm}
Let $h\!: Y\!\lra\!X$ be a smooth map and $k\!\in\!\Z^{\ge0}$.
For every closed subset \hbox{$A\!\subset\!X$} such that $A\!\subset\!h(N_k(h)\!)$
and an open neighborhood $W\!\subset\!X$ of~$A$,
there exists an open neighborhood $U\!\subset\!W$ of $A$ such~that
\BE{neighb2_e0}\Hcl_{\{U\};l}(X;R)=0 \qquad\hbox{if}~~l\!>\!k.\EE
\end{lmm}

\begin{proof} Let $n\!=\!\dim X$.
Since the open subsets $X\!-\!A,W\!\subset\!X$ cover~$X$, 
there exists a triangulation (e.g. by repeated subdivision)
$T\!=\!(K,\eta)$ of~$X$ such that 
the image of every simplex~$\si\!\in\!K$ is contained either in $X\!-\!A$ or in~$W$.
By the proof of \cite[Theorem~1]{Z2}, we can also assume that
the smooth map~$h$ is transverse to $\eta|_{\Int\si}$  for every $\si\!\in\!K$. 
In particular, a dimension count shows
$$h\big(N_k(h)\!\big)\subset
\eta\big(|K|\!-\!|K_{n-1-k}|\big)=
\!\bigcup_{\begin{subarray}{c}\si\in K\\ \dim\si\ge n-k\end{subarray}}
\!\!\!\!\!\!\!\!\eta(\Int\si).$$
Since $A\!\subset\!h(N_k(h)\!)$, it follows~that
$$A\subset U\!\equiv\!\!
\bigcup_{\begin{subarray}{c}\si\in K\\ \dim\si\ge n-k\\ 
\eta(|\si|)\cap A\neq\eset\end{subarray}}\!\!\!\!\!\!\!\!\eta(\Int\si)
=\bigcup_{\begin{subarray}{c}\si\in K\\ \dim\si\ge n-k\\ 
\eta(|\si|)\cap A\neq\eset\end{subarray}}\!\!\!\!\!\!\!\!\eta\big(\St(b_{\si},\sd K)\!\big)
\subset\!
\bigcup_{\begin{subarray}{c}\si\in K\\ \dim\si\ge n-k\\ 
\eta(|\si|)\cap A\neq\eset\end{subarray}}
\!\!\!\!\!\!\!\!\eta\big(|\si|\big)\subset W\,.$$
We show below that the open neighborhood $U\!\subset\!W$ of $A$ satisfies~\e_ref{neighb2_e0},
adapting the proof of \cite[Lemma~2.4]{Z}.\\

\noindent
For each $m\!\in\![n]$, let 
$$U_m=\bigcup_{\begin{subarray}{c}\si\in K\\ \dim\si=m\\ 
\eta(\si)\cap A\neq\eset\end{subarray}}
\!\!\!\!\!\!\!\eta\big(\St(b_{\si},\sd K)\!\big)\subset W.$$
For $m_1,\ldots,m_j\!\in\![n]$ with $m_1\!<\!\ldots\!<\!m_j$, let
$$\cA_{m_1\ldots m_j}=
\big\{\!(\si_1,\ldots,\si_j)\!\in\!K^j\!:
\si_1\!\subset\!\ldots\!\subset\!\si_j,~\dim\si_1\!=\!m_1,~\ldots,~\dim\si_j\!=\!m_j,
~\eta(\si_1)\!\cap\!A\!\neq\!\eset\big\}.$$
We note that 
\begin{alignat*}{2}
\St(b_{\si},\sd K)\cap\St(b_{\si'},\sd K)&=\eset
&\qquad&\hbox{if}~~\si\!\not\subset\!\si'~\hbox{and}~\si\!\not\supset\!\si'\,,\\
\St(b_{\si_1},\sd K)\cap\ldots\cap\St(b_{\si_j},\sd K)&=\St(b_{\si_1}\ldots b_{\si_j},\sd K)
&\qquad&\hbox{if}~\si_1\!\subset\!\ldots\!\subset\!\si_j\,.
\end{alignat*}
Thus, every intersection $U_{m_1}\!\cap\!\ldots\!\cap\!U_{m_j}$ with $m_1\!<\!\ldots\!<\!m_j$
is a disjoint union of the open stars $\eta(\St(b_{\si_1}\ldots b_{\si_j},\sd K)\!)$
with $(\si_1,\ldots,\si_j)\!\in\!\cA_{m_1\ldots m_j}$.\\

\noindent
Since the collection $\eta(\St(b_{\si_1}\ldots b_{\si_j},\sd K)\!)$
with $(\si_1,\ldots,\si_j)\!\in\!\cA_{m_1\ldots m_j}$ is locally finite in~$X$ and
consists of disjoint subsets, \e_ref{SclWdisj_e} gives
\BE{neighb2_e7} 
\Hcl_{\{U_{m_1}\cap\ldots\cap U_{m_j}\};*}(X;R)=
\prod_{(\si_1,\ldots,\si_j)\in\cA_{m_1\ldots m_j}}
\hspace{-.45in}\Hcl_{\{\eta(\St(b_{\si_1}\ldots b_{\si_j},\sd K))\};*}(X;R)\EE
for all $m_1,\ldots,m_j\!\in\![n]$ with $m_1\!<\!\ldots\!<\!m_j$.
Since the closure of each contractible subset $\eta(\St(b_{\si_1}\ldots b_{\si_j},\sd K)\!)$
in~$X$ is compact, \e_ref{SclWcmpt_e} gives
$$\Hcl_{\{\eta(\St(b_{\si_1}\ldots b_{\si_j},\sd K)\!)\};l}(X;R)
=H_l\big(\eta(\St(b_{\si_1}\ldots b_{\si_j},\sd K)\!);R\big)=0
~~\forall~l\!\neq\!0,~(\si_1,\ldots,\si_j)\!\in\!\cA_{m_1\ldots m_j}.$$
Combining this with~\e_ref{neighb2_e7}, we obtain 
\BE{neighb2_e9} 
\Hcl_{\{U_{m_1}\cap\ldots\cap U_{m_j}\};l}(X;R)=0 \quad\forall~l\!\ge\!1.\EE
By induction on $j\!=\!1,2,\ldots$, Corollary~\ref{MVSpanier_crl} (Mayer-Vietoris) 
and~\e_ref{neighb2_e9} give
$$\Hcl_{\{U_{m_1}\cup\ldots\cup U_{m_j}\};l}(X;R)=0 \quad\forall~l\!\ge\!j.$$
Since $U=\!U_{n-k}\!\cup\!\ldots\!\cup\!U_n$, this gives~\e_ref{neighb2_e0}.
\end{proof}

\subsection{Oriented Borel-Moore homology}
\label{OrientedBM_subs}

\noindent
The construction of the oriented singular chain complex~$\ov{S}_*(X;\Z)$ in \cite[Section~2.3]{Z}
readily extends to locally finite chains.
Cycles are much easier to construct in the resulting quotient chain complexes~$\ovScl_*(X;R)$
and~$\ovScl_*(X,\{U\};R)$.
By Proposition~\ref{isom1_prp} below, the homologies~$\ovHcl_*(X;R)$ of $\ovScl_*(X;R)$
and~$\ovHcl_*(X,\{U\};R)$ of~$\ovScl_*(X,\{U\};R)$  
are naturally isomorphic to~$\Hcl_*(X;R)$ and $\Hcl_*(X,\{U\};R)$, respectively.\\

\noindent
For $k\!\in\!\Z^{\ge0}$ and $\tau\!\in\!\cS_k$, let
$$\ti\tau=\Id_{\De^k}-(\sign\tau)\tau\in S_k(\De^k;R).$$
For a topological space~$X$, let 
$$S_k'(X;R)\subset S_k(X;R)$$
be the $R$-submodule generated by the chains $\si_{\#}(\ti\tau)\!\in\!S_k(X;R)$
with $\si\!\in\!\Hom(\De^k,X)$ and $\tau\!\in\!\cS_k$.
In the notation~\e_ref{BMch_e}, define
$$\prScl_k(X;R)=
\Big\{\sum_{\si\in\Hom(\De^k,X)}\sum_{\tau\in\cS_k}
\!\!a_{\si,\tau}\si_{\#}(\ti\tau)\!\in\!\Scl_k(X;R)\!:a_{\si,\tau}\!\in\!R\Big\}.$$
In the perspective of~\e_ref{cmapdfn_e}, $\prScl_k(X;R)$ consists of the singular chains
$c\!\in\!\Scl_k(X;R)$  such~that
$$c\big|_{\cS_k\si}\in\big\{\si_{\#}(c')|_{\cS_k\si}\!:c'\!\in\!S_k'(\De^k;R)\big\}
~~\forall\,\si\!\in\!\Hom(\De^k,X), \quad
\hbox{where}~~\cS_k\si\equiv\{\si\!\circ\!\tau\!:\tau\!\in\!\cS_k\}.$$
If in addition $U\!\subset\!X$, let
$$\prScl_{\{U\};*}(X;R)=\Scl_{\{U\};*}(X;R)\!\cap\!\prScl_*(X;R).$$

\vspace{.2in}

\noindent
By \cite[Lemma~2.6]{Z}, $\prt\wt\tau\!\in\!S'_{k-1}(\De^k)$
for all $\tau\!\in\!\cS_k$ and
$k\!\in\!\Z^{\ge0}$.
Thus, $\prScl_*(X;R)$ is a subcomplex of $(\Scl_*(X;R),\prt_X)$ and 
$\prScl_{\{U\};*}(X;R)$ is a subcomplex of~$(\Scl_{\{U\};*}(X;R),\prt_X)$.
Let
\begin{gather*}
\ovScl_*(X;R)=\frac{\Scl_*(X;R)}{\prScl_*(X;R)}, \qquad
\ovScl_{\{U\};*}(X;R)=\frac{\Scl_{\{U\};*}(X;R)}{\prScl_{\{U\};*}(X;R)}
\subset \ovScl_*(X;R), \\
\ovScl_*(X,\{U\};R)=\frac{\ovScl_*(X;R)}{\ovScl_{\{U\};*}(X;R)}\,.
\end{gather*}
We denote the image of a Borel-Moore singular chain $c\!\in\!\Scl_k(X;R)$
in~$\ovScl_k(X;R)$ by~$\{c\}$,
the induced boundary operator on~$\ovScl_k(X;R)$ by~$\ov\prt_X$, 
and the homologies of the above three chain complexes by
$\ovHcl_*(X;R)$, $\ovHcl_{\{U\};*}(X;R)$, and $\ovHcl_*(X,\{U\};R)$,
respectively.
The quotient projection maps on the chain complexes induce homomorphisms
\BE{IsomHom_e}\begin{split} &\Hcl_*(X;R)\lra\ovHcl_*(X;R), \qquad
\Hcl_{\{U\};*}(X;R)\lra\ovHcl_{\{U\};*}(X;R), \\
&\hspace{.8in}
\Hcl_*\big(X,\{U\};R\big)\lra\ovHcl_*\big(X,\{U\};R\big).
\end{split}\EE

\vspace{.2in}

\noindent
If $h\!: X\!\lra\!Y$ is a proper continuous map between topological spaces
and \hbox{$f(U)\!\subset\!W\!\subset\!Y$}, 
the induced homomorphism 
$$h_{\#}\!: \Scl_*(X;R)\lra S_*(Y;R)$$
takes $\prScl_{\{U\};*}(X;R)$ into $\prScl_{\{W\};*}(Y;R)$.
Thus, $h_{\#}$ induces homomorphisms
\begin{equation*}\begin{split} &h_*\!:\ovHcl_*(X;R)\lra\ovHcl_*(Y;R), \qquad
h_*\!:\ovHcl_{\{U\};*}(X;R)\lra\ovHcl_{\{W\};*}(Y;R), \\
&\hspace{.8in}
h_*\!:\ovHcl_*\big(X,\{U\};R\big)\lra\ovHcl_*\big(Y,\{W\};R\big).
\end{split}\end{equation*}

\begin{prp}\label{isom1_prp}
For any topological space~$X$, the homomorphisms~\e_ref{IsomHom_e} are isomorphisms.
\end{prp}

\begin{proof} The natural transformation of functors $D_X\!: S_*\!\lra\! S_{*+1}$ 
provided by \cite[Lemma~2.7]{Z} satisfies
\BE{isom1_e3}D_X\big(S_k'(X;R)\!\big)\subset S_{k+1}'(X;R) \quad\hbox{and}\quad
\prt_XD_X\big|_{S_k'(X;R)}=\big\{\!(-1)^{k+1}\Id+D_X\prt_X\!\big\}\!\big|_{S_k'(X;R)}\,.\EE
Define
$$\hb\!:\Hom(\De^k,X)\lra S_{k+1}(\De^k;R), \quad \hb(\si)=D_{\De^k}\big(\id_{\De^k}\big).$$
By the naturality of $D_X$ (or \cite[(2.11)]{Z} and (\ref{hmap_e2b})),
$$D_X\!=\!\hb_{\#}\!: S_k(X;R)\lra S_{k+1}(X;R).$$
By Lemma~\ref{hrigid_lmm}, $D_X$ thus extends to a homomorphism
$$D_X\!=\!\hb_{\#}\!: \Scl_k(X;R)\lra \Scl_{k+1}(X;R),$$
which is natural with respect to proper continuous maps.
By~\e_ref{isom1_e3}, 
\BE{isom1_e7}
D_X\big(\prScl_k(X;R)\!\big)\subset\prScl_{k+1}(X;R) \quad\hbox{and}\quad
\prt_XD_X\big|_{\prScl_k(X;R)}=\big\{\!(-1)^{k+1}\Id+D_X\prt_X\!\big\}\!\big|_{\prScl_k(X;R)}\,.\EE
Thus, all homology groups of the chain complex $(\prScl_*(X;R),\prt_X|_{\prScl_*(X;R)})$ vanish.
Combining this with the homology long exact sequence for the exact sequence of chain complexes
$$0\lra \prScl_*(X;R)\lra \Scl_*(X;R)\lra\ovScl_*(X;R)\lra 0,$$
we conclude that the first homomorphism in~\e_ref{IsomHom_e} is an isomorphism.\\

\noindent
Since $D_X(\Scl_{\{U\};k}(X;R)\!)\!\subset\!\Scl_{\{U\};k+1}(X;R)$,
$$D_X\big(\prScl_{\{U\};k}(X;R)\!\big)\subset\prScl_{\{U\};k+1}(X;R) ~~\hbox{and}~~
\prt_XD_X\big|_{\prScl_{\{U\};k}(X;R)}=
\big\{\!(-1)^{k+1}\Id+D_X\prt_X\!\big\}\!\big|_{\prScl_{\{U\};k}(X;R)}.$$
Along with the second statement in~\e_ref{isom1_e7} and 
the homology long exact sequence for the exact sequence of chain complexes
$$0\lra \prScl_{\{U\};*}(X;R)\lra \Scl_{\{U\};*}(X;R)\lra
\ovScl_{\{U\};*}(X;R)\lra 0,$$
this implies that the second homomorphism in~\e_ref{IsomHom_e} is an isomorphism.
The claim for the third homomorphism in~\e_ref{IsomHom_e} follows from 
the homology long exact sequence for the exact sequence of chain complexes
$$0\lra \ovScl_{\{U\};*}(X;R)\lra \ovScl_*(X;R)\lra
\ovScl_*(X,\{U\};R)\lra 0,$$
the claims for the first two homomorphisms, and the Five Lemma.
\end{proof}

\noindent
If $X$ is a manifold,
the operator~$D_X$ of \cite[Lemma~2.7]{Z} sends smooth maps 
into linear combinations of smooth maps. 
Thus, the above constructions go through for the chain complexes based 
on elements in $C^{\i}(\De^k,X)$ instead of~$\hbox{Hom}(\De^k,X)$.
The two chain complexes define the same homology groups of~$X$
by Whitney Approximation Theorem \cite[Theorem~6.26]{Lee}.
In Sections~\ref{psi_sec}-\ref{isom_sec}, all chain complexes and homology groups
are based on smooth maps.\\

\noindent
From now on, we restrict the coefficient ring~$R$ to~$\Z$.
We call a tuple $(\si_i)_{i\in\cI}$ of elements of~$\hbox{Hom}(\De^k,X)$
\sf{locally finite} if for every $x\!\in\!X$ there exists an open neighborhood
$U_x\!\subset\!X$ so that the set
$$\ale_{(\si_i)_{i\in\cI}}(U_x)\equiv
\big\{i\!\in\!\cI\!:\si_i(\De^k)\!\cap\!U_x\!\neq\!\eset\big\}$$
is finite.
For any such collection,
\BE{Zsingch_e}c\equiv \sum_{i\in\cI}\si_i\in\Scl_k(X;\Z).\EE
If $k\!\in\!\Z^+$, every element of $\ovScl_k(X;\Z)$ can be represented by 
a chain as in~\e_ref{Zsingch_e} for some locally finite tuple 
$(\si_i)_{i\in\cI}$ of elements of~$\hbox{Hom}(\De^k,X)$.\\

\noindent
For~$c$ in~\e_ref{Zsingch_e}, let
$$\cB_c=\big\{\!(i,p)\!: i\!\in\!\cI,~ p\!\in\![k]\!\big\}.$$
Lemmas~\ref{gluing_l1} and~\ref{gluing_l2} below will be used to glue 
the summands in chains~$c$ as in~\e_ref{Zsingch_e} that represent cycles and 
bounding chains in~$\ovScl_*(X;\Z)$ into smooth maps from manifolds.
The two lemmas are the direct extensions of Lemmas~2.10 and~2.11 in~\cite{Z}
to the Borel-Moore chains.
They hold for the same reasons because the local finiteness conditions
implies that each boundary simplex $\si_i\!\circ\!\io_{k;p}$
with $(i,p)\!\in\!\cB_c$ appears only finitely many times in~$\prt_Xc$.

\begin{lmm}\label{gluing_l1}
If $k\!\in\!\Z^+$ and the chain \e_ref{Zsingch_e} determines a cycle 
in~$\ovScl_k(X;\Z)$,
there exist a subset $\D_c\!\subset\!\cB_c\!\times\!\cB_c$
disjoint from the diagonal and a~map 
$$\tau\!: \D_c\lra \cS_{k-1}, \qquad 
\big(\!(i_1,p_1),(i_2,p_2)\!\big)\lra\tau_{(i_1,p_1),(i_2,p_2)},$$ 
with the following properties:
\begin{enumerate}[label=(\roman*),leftmargin=*]

\item if $(\!(i_1,p_1),(i_2,p_2)\!)\!\in\!\D_c$, then
$(\!(i_2,p_2),(i_1,p_1)\!)\!\in\!\D_c$;

\item the projection $\D_c\!\lra\!\cB_c$ on either coordinate is a bijection;

\item for all $(\!(i_1,p_1),(i_2,p_2)\!)\!\in\!\D_c$, 
\begin{gather}\label{psi_e1a}
\tau_{(i_1,p_1),(i_2,p_2)}^{~-1}=\tau_{(i_2,p_2),(i_1,p_1)}, \qquad
\si_{i_1}\!\circ\!\io_{k;p_1}\!\circ\!\tau_{(i_1,p_1),(i_2,p_2)}
=\si_{i_2}\!\circ\!\io_{k;p_2},\\
\label{psi_e1b}
\hbox{and}\qquad \sign\tau_{(i_1,p_1),(i_2,p_2)}=-(-1)^{p_1+p_2}.
\end{gather}

\end{enumerate}
\end{lmm}

\begin{lmm}\label{gluing_l2}
Suppose $k\!\ge\!1$, $(\si_{0;i})_{i\in\cI_0}$ and $(\si_{1;i})_{i\in\cI_1}$
are locally finite tuples of elements of $\hbox{Hom}(\De^k,X)$,
$(\wt\si_i)_{i\in\wt\cI}$
is a locally finite tuple of elements of~$\hbox{Hom}(\De^{k+1},X)$, and
\BE{gluing2_e}c_0\equiv\sum_{i\in\cI_0}\si_{0;i}, \quad
c_1\equiv\sum_{i\in\cI_1}\si_{1;i}, \quad
\wt{c}\equiv\sum_{i\in\wt\cI}\wt\si_i, \quad
\ov\prt\{\ti{c}\}=\{c_1\}\!-\!\{c_0\}\in\ovScl_k(X;\Z).\EE
Then there exist a subset $\D_{\ti{c}}\!\subset\!\cB_{\wt{c}}\!\times\!\cB_{\wt{c}}$
disjoint from the diagonal, disjoint subsets
\hbox{$\cB_{\ti{c}}^{(0)},\cB_{\ti{c}}^{(1)}\!\subset\!\cB_{\ti{c}}$}, 
and~maps
\begin{gather*}
\ti\tau\!: \D_{\ti{c}}\lra \cS_k, \quad 
\big(\!(i_1,p_1),(i_2,p_2)\!\big)\lra\ti\tau_{(i_1,p_1),(i_2,p_2)},\\
(\ti\io_r,\ti{p}_r)\!:\cI_r\lra\cB_{\ti{c}}^{(r)},
~~~\hbox{and}~~~
\ti\tau_r\!:\cI_r\lra\cS_k, ~~i\lra\ti\tau_{(r,i)}, \quad r=0,1,
\end{gather*}
with the following properties:
\begin{enumerate}[label=(\roman*),leftmargin=*]

\item if $(\!(i_1,p_1),(i_2,p_2)\!)\!\in\!\D_{\wt{c}}$, then
$(\!(i_2,p_2),(i_1,p_1)\!)\!\in\!\D_{\wt{c}}$;

\item the projection $\D_{\ti{c}}\!\lra\!\cB_{\ti{c}}$  on either coordinate
is a bijection onto the complement of~$\cB_{\ti{c}}^{(0)}\!\cup\!\cB_{\ti{c}}^{(1)}$;

\item for all $(\!(i_1,p_1),(i_2,p_2)\!)\!\in\!\D_{\ti{c}}$, 
\begin{gather}\label{psi_e2a}
\ti\tau_{(i_1,p_1),(i_2,p_2)}^{~-1}=\ti\tau_{(i_2,p_2),(i_1,p_1)}, \qquad
\wt\si_{i_1}\!\circ\!\io_{k+1;p_1}\!\circ\!\ti\tau_{(i_1,p_1),(i_2,p_2)}
=\wt\si_{i_2}\!\circ\!\io_{k+1;p_2},\\
\label{psi_e2b}
\hbox{and}\qquad \sign\ti\tau_{(i_1,p_1),(i_2,p_2)}=-(-1)^{p_1+p_2};
\end{gather}

\item for all $r\!=\!0,1$ and $i\!\in\!\cA_r$, 
\BE{psi_e3}
\ti\si_{\ti\io_r(i)}\!\circ\!\io_{k+1;\ti{p}_r(i)}\!\circ\!\ti\tau_{(r,i)}=\si_{r;i}
\qquad\hbox{and}\qquad \sign\ti\tau_{(r,i)}=-(-1)^{r+\ti{p}_r(i)};\EE

\item $(\ti\io_r,\ti{p}_r)$ is a bijection onto $\cB_{\wt{c}}^{(r)}$ for $r\!=\!0,1$.

\end{enumerate}
\end{lmm}

\vspace{.2in}

\noindent
Suppose $\ov{V}$ is an oriented $k$-manifold with boundary and
$(K,\eta)$ is a triangulation of~$\ov{V}$ that restricts to a triangulation of~$\prt\ov{V}$.
Let
$$K^{\top}=\big\{\si\!\in\!K\!:\dim\si\!=\!k\big\}.$$
For each $k$-dimensional simplex $\si\!\in\!K$, let
\BE{lsidfn_e}l_{\si}\!: \De^k\lra\si\subset|K|\subset\R^{\i}\EE
be a linear map such that the composition $\eta\!\circ\!l_{\si}$ is orientation-preserving.
The fundamental class \hbox{$[\ov{V}]\!\in\!\ovHcl_k(\ov{V},\prt\ov{V};\Z)$} of~$M$
is then represented~by
$$\sum_{\si\in K^{\top}}\!\!\!\!\big\{\eta\!\circ\!l_{\si}\big\}
\in\ovScl_k\big(\ov{V},\{\prt\ov{V}\};\Z\big).$$
The corresponding sum
$$\sum_{\si\in K^{\top}}\!\!\!\!\eta\!\circ\!l_{\si}\in\Scl_k\big(\ov{V},\{\prt\ov{V}\};\Z\big) $$
may not be a cycle.
If $f\!:\ov{V}\!\lra\!X$ is a proper map and $U\!\subset\!X$ 
is a subset containing $f(\prt\ov{V})$,
then $f_*([\ov{V}])\!\in\!\ovHcl_k(X,\{U\};\Z)$ is represented
by 
$$\sum_{\si\in K^{\top}}\!\!\!\!\big\{f\!\circ\!\eta\!\circ\!l_{\si}\big\}
\in \ovScl_k\big(X,\{U\};\Z\big);$$
by the properness of $f$, the collection 
$\{f\!\circ\!\eta\!\circ\!l_{\si}\}_{\si\in K^{\top}}$
is locally finite in~$X$.

\subsection{From integral cycles to pseudocycles}
\label{psi_sec}

\noindent
In this section, we extend the constructions of \cite[Section~3.1]{Z}
from finite to locally finite singular chains and obtain the first homomorphism 
in~\e_ref{mainthm_e}.
We start with a cycle \hbox{$\{c\}\!\in\!\ovScl_k(X;\Z)$} as in Lemma~\ref{gluing_l1}
and replace each singular simplex~$\si_i$  by its composition 
with the self-map~$\vp_k$ of~$\De^k$ provided by Lemma~\ref{bdpush_lmm}. 
The functions~$\si\!\circ\!\vp_k$ still satisfy the second equation in~\e_ref{psi_e1a}, i.e.
\BE{pertmaps_e1}
\si_{i_1}\!\circ\!\vp_k\!\circ\!\io_{k;p_1}\!\circ\!\tau_{(i_1,p_1),(i_2,p_2)}
=\si_{i_2}\!\circ\!\vp_k\!\circ\!\io_{k;p_2}
\quad\forall~\big(\!(i_1,p_1),(i_2,p_2)\!\big)\!\in\!\D_c,\EE
because $\vp_k$ restricts to the identity on the boundary of~$\De^k$. 
This allows us to glue the maps $\si_i\!\circ\!\vp_k$ into a proper map~$F$
from a $k$-dimensional simplicial complex~$M$ to~$X$.
Removing the codimension~2 simplicies, we obtain a Borel-Moore pseudocycle
in the proof of Lemma~\ref{constr1_lmm}.
In the proof of Lemma~\ref{equiv1_lmm},
we use a similar procedure 
to turn a bounding chain \hbox{$\{\wt{c}\}\!\in\!\ovScl_{k+1}(X;\Z)$}
into a Borel-Moore pseudocycle equivalence between the Borel-Moore pseudocycles
determined by its boundaries.

\begin{lmm}\label{constr1_lmm}
Let $X$ be a manifold and $k\!\in\!\Z^{\ge0}$.
Every integer locally finite singular $k$-chain~$c$ as in~\e_ref{Zsingch_e}
with $\si_i\!\in\!C^{\i}(\De^k;X)$ for all $i\!\in\!\cI$ representing a cycle
in~$\ovScl_k(X;\Z)$ determines an element of~$\cHcl_k(X)$.
\end{lmm}

\begin{proof} If $k\!=\!0$, $(\si_i)_{i\in\cI}$ is a discrete collection of points of~$X$.
Thus,
$$F\!: M\!\equiv\!M'\!\equiv\!\cI\lra X, \qquad F(i)=\si_i(0),$$
is a Borel-Moore 0-pseudocycle in~$X$.\\

\noindent
Suppose $k\!\ge\!1$. Let 
$$\D_c\subset\cB_c\!\times\!\cB_c \qquad\hbox{and}\qquad
\tau\!:\D_c\!\lra\!\cS_{k-1}$$
be the subset and map corresponding to~$c$ as in Lemma~\ref{gluing_l1}.
Define
\begin{gather}\label{Mprdfn_e}
M'=\Big(\bigsqcup_{i\in\cI}\{i\}\!\times\!\De^k\Big)\!\!\Big/\!\!\!\sim,
\qquad\hbox{where}\\
\notag 
\big(i_1,\io_{k;p_1}(\tau_{(i_1,p_1),(i_2,p_2)}(t)\!)\!\big)
\sim\big(i_2,\io_{k;p_2}(t)\!\big)
\quad\forall~
\big(\!(i_1,p_1),(i_2,p_2)\!\big)\!\in\!\D_c,~t\!\in\!\De^{k-1}.
\end{gather}
Let $\pi$ be the quotient map and 
\BE{fmapdfn_e}
F\!: M'\lra X, \qquad F\big([i,t]\big)=\si_i\big(\!\vp_k(t)\!\big)
~~\forall~i\!\in\!\cI,\,t\!\in\!\De^{k+1}.\EE
This map is well-defined by~\e_ref{pertmaps_e1} and continuous by 
the universal property of the quotient topology.\\

\noindent
Since the maps $\tau_{(i_1,p_1),(i_2,p_2)}$ are linear automorphisms of~$\De^{k-1}$,
$M'$ is homeomorphic to a geometric realization of a simplicial complex (cf. \cite[Section 3]{Mu2}).
Thus, $M'$ is a Hausdorff topological space, and $\pi$ is a closed~map.
By the local finiteness of $(\si_i)_{i\in\cI}$, the~set
$$\big\{i\!\in\!\cI\!:F\big(\pi(\{i\}\!\times\!\De^k)\!\big)\!\cap\!A\!\neq\!\eset\big\}
=\ale_{(\si_i)_{i\in\cI}}(A)$$
is finite for every compact subset $A\!\subset\!X$.
Since $\pi(\{i\}\!\times\!\De^k)\!\subset\!M'$ is compact as well,
it follows that $F$ is a proper~map.
Since $X$ is second countable, $\cI$ is countable, and thus $M'$ is second countable.\\

\noindent
With $Y\!\subset\!\De^k$ denoting the $(k\!-\!2)$-skeleton,
let $M\!\subset\!M'$ be the complement of the subset 
\BE{Ycdfn_e}Y_c\equiv \pi\Big(\bigsqcup_{i\in\cI}\{i\}\!\times\!Y\Big)\subset M'.\EE
Since $M'$ is Hausdorff, $Y_c\!\subset\!M'$ is closed, 
and $F$ is a proper~map, Lemma~\ref{proper_lmm}\ref{Bdext_it}\ref{proper4} gives
\BE{constr1_e3}
\Bd F|_M\subset F(Y_c)=\bigcup_{i\in\cI}\si_i\big(\!\vp_k(Y)\!\big)
= \bigcup_{i\in\cI}\si_i(Y)\,;\EE
the last equality holds by the first equation in~\e_ref{bdpush_e1}.
Since $\si_i|_{\Int\De'}$ is smooth for all $i\!\in\!\cI$ and 
all simplices $\De'\!\subset\!\De^k$,
$\Bd F|_M$ has dimension at most $k\!-\!2$ by~\e_ref{constr1_e3}.\\

\noindent
By the above, $F|_M$ is a Borel-Moore $k$-pseudocycle, provided 
$M$ is an oriented manifold and $F|_M$ is a smooth~map.
These are local statements, and~(2) in the proof of \cite[Lemma~3.2]{Z}
applies {\it verbatim}.
\end{proof}

\begin{lmm}\label{equiv1_lmm}
Let $X$ be a manifold and $k\!\in\!\Z^{\ge0}$.
Suppose $c_0,c_1$ are integer locally finite singular $k$-chains as in~\e_ref{gluing2_e}
with \hbox{$\si_{r;i}\!\in\!C^{\i}(\De^k;X)$} for all $i\!\in\!\cI_r$ representing cycles
in~$\ovScl_k(X;\Z)$ and $(M_r',M_r,F_r)$ with $r\!=\!0,1$ are the triples corresponding 
to~$c_0,c_1$ via the construction of Lemma~\ref{constr1_lmm}. 
Every integer locally finite singular $(k\!+\!1)$-chain~$\wt{c}$ as in~\e_ref{gluing2_e}
with \hbox{$\wt\si_i\!\in\!C^{\i}(\De^{k+1};X)$} for all $i\!\in\!\wt\cI$ 
determines a Borel-Moore pseudocycle equivalence between 
the pseudocycles~$F_0|_{M_0}$ and~$F_1|_{M_1}$.
\end{lmm}

\begin{proof}
If $k\!=\!0$, there are subsets $\D_{\wt{c}}\!\subset\!\wt\cI\!\times\!\wt\cI$ and 
$\cI_{\wt{c}}^{(0)},\cI_{\wt{c}}^{(1)}\!\subset\!\wt\cI$ and bijections
$$\wt\io_r\!:\cI_r\lra\cI_{\wt{c}}^{(r)}, \qquad r\!=\!0,1,$$
such that the projections
$$\D_{\wt{c}}\lra \wt\cI\!-\!\cI_{\wt{c}}^{(1)} \quad\hbox{and}\quad
\D_{\wt{c}}\lra \wt\cI\!-\!\cI_{\wt{c}}^{(0)}$$
on the first and second component, respectively, are bijections,
$$\wt\si_{i_1}(1)=\wt\si_{i_2}(0) ~~\forall~(i_1,i_2)\!\in\!\D_{\wt{c}},
\qquad
\wt\si_{\wt\io_r(i)}(r)=\si_{r;i}(0)~~\forall~i\!\in\!\cI_r,\,r\!=\!0,1.$$
The space 
$$\wt{M}=\Big(\bigsqcup_{i\in\wt\cI}\{i\}\!\times\!\De^1\Big)\!\!\Big/\!\!\!\sim,
\quad\hbox{where}\quad 
(i_1,1)\sim(i_2,0)
~\forall~(i_1,i_2)\!\in\!\D_{\ti{c}},$$
is then an oriented one-dimensional manifold with boundary
$\prt\wt{M}\!=\!M_1\!-\!M_0$.
Similarly to the proof of Lemma~\ref{constr1_lmm}, the map
$$\wt{F}\!:\wt{M}\lra X, \qquad \wt{F}\big([i,t]\big)=\wt\si_i\big(\vp_1(t)\!\big),$$
is well-defined, continuous, proper, and smooth.
Since $\wt{F}|_{M_r}\!=\!F_r$, $\wt{F}$ is a pseudocycle equivalence 
between~$F_0\!=\!F_0|_{M_0}$ and~$F_1\!=\!F_1|_{M_1}$.\\

\noindent
Suppose $k\!\ge\!1$. Let 
$$\D_{\ti{c}}\subset\cB_{\wt{c}}\!\times\!\cB_{\wt{c}}, \quad
\cB_{\ti{c}}^{(0)},\cB_{\ti{c}}^{(1)}\!\subset\!\cB_{\ti{c}}, \quad
\ti\tau\!: \D_{\ti{c}}\lra\cS_k, \quad
(\ti\io_r,\ti{p}_r)\!:\cI_r\lra\cB_{\ti{c}}^{(r)}, \quad
\ti\tau_r\!:\cI_r\lra\cS_k$$
be the subsets and maps corresponding to~$\wt{c}$ as in 
Lemmas~\ref{gluing_l2}.
As detailed in \cite[Section~3.1]{Z},
$\vp_{k+1}\!=\!\id$ on $\prt\De^{k+1}$,
the third equation in~\e_ref{bdpush_e2}, the second equation in~\e_ref{bdpush_e1},
and the first equation in~\e_ref{psi_e3} give
\BE{pertmaps_e2}
\ti\si_{\ti\io_r(i)}\!\circ\!\ti\vp_{k+1}\!\circ\!\vp_{k+1}\!\circ\!
\io_{k+1;\ti{p}_r(i)}\!\circ\!\ti\tau_{(r,i)}=\si_{r;i}\!\circ\!\vp_k
\qquad\forall~i\!\in\!\cI_r,~r\!=\!0,1.\EE
Furthermore, $\vp_{k+1}\!=\!\id$ on $\prt\De^{k+1}$,
the third equation in~\e_ref{bdpush_e2} used twice, 
the second equation in~\e_ref{bdpush_e1}, and
the second equation in~\e_ref{psi_e2a} give
\BE{pertmaps_e3}
\ti\si_{i_1}\!\circ\!\ti\vp_{k+1}\!\circ\!\vp_{k+1}
\!\circ\!\io_{k+1;p_1}\!\circ\!\ti\tau_{(i_1,p_1),(i_2,p_2)}
=\ti\si_{i_2}\!\circ\!\ti\vp_{k+1}\!\circ\!\vp_{k+1}\!\circ\!\io_{k+1;p_2}
~~\forall\,(\!(i_1,p_1),(i_2,p_2)\!)\!\in\!\D_{\ti{c}}.\EE

\vspace{.2in}

\noindent
Define 
\begin{gather*}
\wt{M}'=\Big(\bigsqcup_{i\in\wt\cI}\{i\}\!\times\!\De^{k+1}\Big)\!\!\Big/\!\!\!\sim,
\qquad\hbox{where}\\ 
\big(i_1,\io_{k+1;p_1}(\ti\tau_{(i_1,p_1),(i_2,p_2)}(t)\!)\!\big)
\sim\big(i_2,\io_{k+1;p_2}(t)\!\big)
\quad\forall~
\big(\!(i_1,p_1),(i_2,p_2)\!\big)\!\in\!\ti\D_{\ti{c}},~t\!\in\!\De^k.
\end{gather*}
Let $\ti\pi$ be the quotient map and
$$\ti{F}\!:\ti{M}'\lra X, \qquad
\ti{F}\big([i,t]\big)=\ti\si_i\big(\ti\vp_{k+1}(\vp_{k+1}(t)\!)\!\big)
 ~~\forall~i\!\in\!\wt\cI,\,t\!\in\!\De^{k+1}.$$
This map is well-defined by \e_ref{pertmaps_e3} 
and is continuous by the universal property of the quotient topology.
By the same reasoning as in the proof of Lemma~\ref{constr1_lmm}, 
$\ti{M}$ is a second countable, Hausdorff topological space,
$\wt\pi$ is a closed map, and $\wt{F}$ is a proper~map.\\

\noindent
With  $\ti{Y}\!\subset\!\De^{k+1}$ denoting the $(k\!-\!1)$-skeleton,
let \hbox{$\ti{M}\!\subset\!\ti{M}'$} be the complement of the~subset
$$Y_{\wt{c}}\equiv \ti\pi\Big(\bigsqcup_{i\in\wt\cI}\{i\}\!\times\!\ti{Y}\Big)\subset\wt{M}'.$$
Since $\wt{M}'$ is Hausdorff, $Y_{\wt{c}}\!\subset\!\wt{M}'$ is closed, 
and $\wt{F}$ is a proper~map,
\BE{equiv1_e3}
\Bd \ti{F}|_{\ti{M}}=\ti{F}(Y_{\wt{c}})
= \bigcup_{i\in\wt\cI}\ti\si_i\big(\ti\vp_{k+1}(\vp_{k+1}(\ti{Y})\!)\!\big)
=\bigcup_{i\in\wt\cI}\ti\si_i(\ti{Y}).\EE
Since $\ti\si_i|_{\Int\De'}$ is smooth for all $i\!\in\!\wt\cI$ and 
all simplices $\De'\!\subset\!\De^{k+1}$,
$\Bd\ti{F}|_{\ti{M}}$ has dimension at most $k\!-\!1$ by~\e_ref{equiv1_e3}.\\

\noindent
For $r\!=\!0,1$, let $Y_r\!\subset\!M_r$ denote the union of the images
of the open $(k\!-\!1)$-simplicies of~$\De^k$ under 
the quotient map~$\pi$ in the proof of Lemma~\ref{constr1_lmm}
(this is also the intersection of~$M_r$ with the union of
the images of the closed $(k\!-\!1)$-simplicies of~$\De^k$ under~$\pi$).
The~maps
$$\io_r\!:M_r\!-\!Y_r\lra\wt{M}, \qquad
\io_r\big([i,t]\big)=\big[\wt\io_r(i),\io_{k+1;\wt{p}_r(i)}\big(\wt\tau_{r;i}(t)\!\big)\big]
~~\forall~i\!\in\!\cI_r,\,t\!\in\!\Int\De^k,$$
are well-defined embeddings with disjoint images.
By~\e_ref{pertmaps_e2} and~\e_ref{fmapdfn_e},
$$\wt{F}\!\circ\!\io_r=F|_{M_r-Y_r}\,.$$
Thus, $\ti{F}|_{\ti{M}}$ is a Borel-Moore pseudocycle equivalence between
the Borel-Moore $k$-pseudocycles $F_0|_{M_0}$ and $F_1|_{M_1}$, 
provided $\ti{M}$ is an oriented manifold,  $\ti{F}|_{\ti{M}}$ is a smooth~map,
$\io_0,\io_1$ are smooth embeddings, and
$$\prt\wt{M}=\io_1(M_1\!-\!Y_1)\!\big)\!\sqcup\!-\io_0(M_0\!-\!Y_0).$$
These are straightforward local statements, 
which are established as in~(3) in the proof of \cite[Lemma~3.3]{Z}.
\end{proof}

\subsection{From pseudocycles to integral cycles}
\label{phi_sec}

\noindent
We next adapt the constructions of \cite[Section~3.2]{Z}
from pseudocycles to Borel-Moore pseudocycles and obtain the second homomorphism 
in~\e_ref{mainthm_e}.
As indicated in Section~\ref{outline_subs},
we first define a homology class~$[f]_{X;U}$ of a pseudocycle~$f$ relative
to a nice neighborhood~$U$ provided by Proposition~\ref{neighb2_prp}
and then pull it back to the absolute Borel-Moore homology of the target.
 
\begin{lmm}\label{euler}
Let $X$ be a manifold and $k\!\in\!\Z^{\ge0}$.
Every Borel-Moore $k$-pseudocycle \hbox{$f\!:M\!\lra\!X$} determines 
an element of~$\Hcl_k(X;\Z)$.
\end{lmm}

\begin{proof}
By Proposition~\ref{neighb2_prp}, there exists an open neighborhood~$U\!\subset\!X$
of $\Bd f$ such~that 
$$\ovHcl_{\{U\};l}(X;\Z)=0 \qquad\forall~ l>k\!-\!2.$$ 
Thus, $f|_{M-f^{-1}(U)}$ is a proper map and the homomorphism
\BE{euler_e2} \ovHcl_k\big(X;\Z\big)\lra  \ovHcl_k\big(X,\{U\};\Z\big)\EE
induced by the quotient is an isomorphism.
Let $V\!\subset\!M$ be an open neighborhood of $M\!-\!f^{-1}(U)$ so that 
$f|_{\ov{V}}$ is still proper and $\ov{V}$ is a manifold with boundary.
This manifold inherits an orientation from~$M$ and thus defines a homology class
$$[\ov{V}]\in \ovHcl_k\big(\ov{V},\{\prt\ov{V}\};\Z\big).$$
Put
\BE{euler_e3} [f]_{X;U}=f_*\big([\ov{V}]\big)\in \ovHcl_k\big(X,\{U\};\Z\big)
\stackrel{\e_ref{euler_e2}}{\cong}\ovHcl_k(X;\Z),\EE
where 
\BE{euler_e4}f_*\!: \ovHcl_k\big(\ov{V},\{\prt\ov{V}\};\Z\big)
\lra \ovHcl_k\big(X,\{U\};\Z\big)\EE
is the homology homomorphism induced by the proper map~$f|_{\ov{V}}$.\\

\noindent 
We now show that $[f]_{X;U}$ doesn't depend on the choices of $V$ or $U$.
Suppose $V'\!\subset\!X$ is an open neighborhood of $\ov{V}$ so that 
$f|_{\ov{V}'}$ is also proper and $\ov{V}'$ is a manifold with boundary.
Choose a triangulation of~$\ov{V}'$ extending some triangulation
of $(\prt\ov{V})\!\cup\!(\prt\ov{V}')$;
such a triangulation exists by \cite[Section~16]{Mu2}. 
Since $f(\ov{V}'\!-\!V)\!\subset\!U$, the classes
$$f_*\big([\ov{V}]\big),f_*\big([\ov{V}']\big)\in \ovHcl_k\big(X,\{U\};\Z\big)$$
are represented by cycles that differ by singular simplices lying in~$U$;
see the last paragraph of Section~\ref{OrientedBM_subs}.
It follows that
$$f_*\big([\ov{V}]\big)=f_*\big([\ov{V}']\big)\in \ovHcl_k\big(X,\{U\};\Z\big).$$
Thus, the homology class $[f]_{X;U}$ is independent of the choice of~$V$.\\

\noindent
Suppose $U'\!\subset\!U$ is an another open neighborhood of $\Bd f$.
By the previous paragraph, we can choose~$V$ for~$U$
and~$V'$ for~$U'$ to be the~same. 
Since the isomorphism~\e_ref{euler_e2} is the composition of the isomorphisms
$$\ovHcl_k(X;\Z)\lra \ovHcl_k\big(X,\{U'\};\Z\big)\lra \ovHcl_k\big(X,\{U\};\Z\big)$$ 
induced by inclusions and the homomorphism~\e_ref{euler_e4} is the composition
$$\ovHcl_k\big(\ov{V},\{\prt\ov{V}\};\Z\big)\lra \ovHcl_k\big(X,\{U'\};\Z\big)
\lra \ovHcl_k\big(X,\{U\};\Z\big),$$ 
the homology classes in $\ovHcl_k(X;\Z)$ corresponding to $[f]_{X;U'}$ and  $[f]_{X;U}$ 
are the~same.
Thus, the homology class~$[f]$ in~$\ovHcl_k(X;\Z)$ corresponding to~$[f]_{X;U}$
under the isomorphism~\e_ref{euler_e2} is independent of the choice of~$U$ as~well. 
\end{proof}

\begin{lmm}\label{equiv2_l}
Let $X$ be a manifold and $k\!\in\!\Z^{\ge0}$.
If Borel-Moore $k$-pseudocycles \hbox{$f_0\!:M_0\!\lra\!X$} and 
\hbox{$f_1\!:M_1\!\lra\!X$} are equivalent, then
$$[f_0]=[f_1]\in\Hcl_k(X;\Z).$$
\end{lmm}

\begin{proof}
Let $\wt{f}\!: \wt{M}\!\lra\!X$ be a Borel-Moore pseudocycle equivalence  
between~$f_0$ and~$f_1$ as in Definition~\ref{BMpseudo_dfn}\ref{BMequiv_it}.
By Remark~\ref{McDuff_rmk}, we can assume that $Y_0,Y_1\!=\!\eset$.
By Proposition~\ref{neighb2_prp}, there exists an open neighborhood~$\wt{U}\!\subset\!X$
of $\Bd\wt{f}$ such~that 
$$\ovHcl_{\{\wt{U}\};l}(X;\Z)=0 \qquad\forall~ l>k\!-\!1.$$ 
Thus, $\wt{f}|_{\wt{M}-\wt{f}^{-1}(\wt{U})}$ is a proper map and the homomorphism
\BE{equiveuler_e2} \ovHcl_k\big(X;\Z\big)\lra  \ovHcl_k\big(X,\{\wt{U}\};\Z\big)\EE
induced by inclusion is injective.\\

\noindent
For $r\!=\!0,1$, let $U_r\!\subset\!\wt{U}$
be an open neighborhood of 
$\Bd f_r\!\subset\!\Bd\wt{f}$ such~that 
\BE{equiveuler_e4}\ovHcl_{\{U_r\};l}(X;\Bbb{Z})=0 \qquad\forall~l>k\!-\!2.\EE
Let $V_r\!\subset\!M_r$ be a choice of an open subset for $(f_r,U_r)$
as in  the proof of Lemma~\ref{euler}.
Since the restriction of~$\wt{f}$ to the closed subset
$$B\equiv \big(\wt{M}\!-\!\wt{f}^{-1}\big(\wt{U}\big)\big)\!\cup\!\ov{V}_0\!\cup\!\ov{V}_1\subset\wt{M}$$ 
is proper, Lemma~\ref{proper_lmm}\ref{properext_it} implies that 
there exists a neighborhood~$W\!\subset\!\wt{M}$ of~$B$
so that $\wt{f}|_{\ov{W}}$ is still proper and 
$\ov{W}$ is a manifold with boundary and corners
(with the corners contained in \hbox{$\prt\wt{M}\!-\!\ov{V}_0-\!\ov{V}_1$)}.
We note~that
\BE{equiv2_e5}\wt{f}\big(\prt\ov{W}\!-\!\big(V_0\!\cup\!V_1\big)\big)
=\wt{f}\big((\ov{W}\!-\!W)\!\cup\!(W\!\cap\!(M_0\!\cup\!M_1)\!)\!-\!V_0\!\cup\!V_1\big)
\subset\wt{U}\!\cup\!U_0\!\cup\!U_1=\wt{U}.\EE
For $r\!=\!0,1$, let
\begin{gather*}
\io_{X;r*}\!:\Hcl_k\big(X,\{U_r\};\Z\big)\lra\Hcl_k\big(X,\{\wt{U}\};\Z\big)
\qquad\hbox{and}\\
\io_{\wt{M};r*}\!: \Hcl_k\big(\ov{V}_r,\{\prt\ov{V}_r\};\Z\big)
\lra \Hcl_k\big(\ov{W},\{\prt\ov{W}\!-\!V_0\!\cup\!V_1\};\Z\big) 
\end{gather*}
be the homomorphisms induced by inclusions.\\

\noindent
Choose a triangulation $T\!=\!(K,\eta)$ of
$\ov{W}$ that restricts to triangulations of 
$\ov{V}_0,\prt\ov{V}_0,\ov{V}_1,\prt\ov{V}_1$ and~$\prt\ov{W}$.
Let
$$K^{\top}=\big\{\si\!\in\!K\!:\dim\si\!=\!k\!+\!1\big\}.$$
For $r\!=\!0,1$, put
$$K_r=\big\{\si\!\in\!K\!:\eta(\si)\!\subset\!\ov{V}_r\big\}, \qquad
K_r^{\top}=\big\{\si\!\in\!K_r\!:\dim\si\!=\!k\big\}.$$
For each $\si\!\in\!K^{\top}$ and $\si\!\in\!K_r^{\top}$, let 
$$l_{\si}\!:\De^{k+1}\lra\si\subset|K| \qquad\hbox{and}\qquad 
l_{\si}\!:\De^k\lra\si\subset|K_r|,$$
respectively, be as in~\e_ref{lsidfn_e}.
By our assumptions,
$$\prt\!\!\!\!\sum_{\si\in K^{\top}}\!\!\!\!\!\big\{\eta\!\circ\!l_{\si}\big\}
+\sum_{r=0,1}\!(-1)^r\!\!\!\!\!\sum_{\si\in K_r^{\top}}\!\!\!\!\!\big\{\eta\!\circ\!l_{\si}\big\}
\in \Scl_{\{\prt\ov{W}-V_0\cup V_1\};k}\big(\wt{M};\Z\big).$$
Along with~\e_ref{equiv2_e5}, this gives
\BE{equiv2_e9}
\prt\!\!\!\!\!\sum_{\si\in K^{\top}}\!\!\!\!\big\{\wt{f}\!\circ\!\eta\!\circ\!l_{\si}\big\}
=\sum_{\si\in K_1^{\top}}\!\!\!\!\!\big\{f_1\!\circ\!\eta\!\circ\!l_{\si}\big\}
-\sum_{\si\in K_0^{\top}}\!\!\!\!\!\big\{f_0\!\circ\!\eta\!\circ\!l_{\si}\big\}
\in \ovScl_l\big(X,\{\wt{U}\};\Z\big).\EE

\vspace{.2in}

\noindent
For $r\!=\!0,1$, let $[f_r]_{X;U_r}\!\in\!\Hcl_k(X,\{U_r\};\Z)$ be as in the proof
of Lemma~\ref{euler} and
$$[f_r]_{X;\wt{U}}=\io_{X;r*}\big([f_r]_{X;U_r}\big)\in \Hcl_k\big(X,\ti{U};\Z\big).$$
Since the diagram
$$\xymatrix{ \Hcl_k(\ov{V}_r,\{\prt\ov{V}_r\};\Z) \ar[rr]^{\{\ti{f}|_{\ov{V_r}}\}_*}
\ar[d]_{\io_{\wt{V};r*}}&&
\Hcl_k(X,\{U_r\};\Z)\ar[d]^{\io_{X;r*}}\\
\Hcl_k(\ov{W},\{\prt\ov{W}\!-\!V_0\!\cup\!V_1\};\Z) 
\ar[rr]^<<<<<<<<<{\{\ti{f}|_{\ov{W}}\}_*}&&\Hcl_k(X,\{\wt{U}\};\Z)}$$
commutes,
$$[f_r]_{X;\wt{U}}=\big\{\ti{f}|_{\ov{W}}\big\}_*
\big(\io_{X;r*}([\ov{V}_r])\!\big)\in\Hcl_k\big(X,\ti{U};\Z\big).$$
By the last paragraph of Section~\ref{OrientedBM_subs},
the first term and the second term on the right-hand side of~\e_ref{equiv2_e9}
represent~$[f_1]_{X;\wt{U}}$ and~$[f_0]_{X;\wt{U}}$, respectively.
Thus,  
$$\io_{X;0*}\big([f_0]_{X;U_0}\big)=\io_{X;1*}\big([f_1]_{X;U_1}\big)
\in \Hcl_k\big(X,\ti{U};\Z\big).$$
Since the diagram 
$$\xymatrix{ \Hcl_k(X;\Z) \ar[rr]^{\e_ref{equiveuler_e4}}_{\cong}
\ar[d]_{\e_ref{equiveuler_e4}}^{\cong} \ar[rrd]|{\e_ref{equiveuler_e2}}&& 
\Hcl_k(X,\{U_0\};\Z)\ar[d]^{\io_{X;0*}}\\
\Hcl_k(X,\{U_1\};\Z) 
\ar[rr]^<<<<<<<<<{\{\io_{X;1*}\}_*}&&\Hcl_k(X,\{\wt{U}\};\Z)}$$
of homomorphisms induced by inclusions commutes and the diagonal homomorphism is injective,
the classes $[f_0],[f_1]\!\in\!\Hcl_k(X;\Z)$ corresponding to 
$[f_0]_{X;U_0}$ and $[f_1]_{X;U_1}$ are the same.
\end{proof}

\subsection{Isomorphisms of homology theories}
\label{isom_sec}

\noindent
In order to establish that the homomorphisms of Theorem~\ref{main_thm} as
constructed in Section~\ref{psi_sec} and~\ref{phi_sec} are isomorphisms
and mutual inverses, we first show~that
$$\Phi_*\!\circ\!\Psi_*\!=\!\id:  \Hcl_*(X;\Z)\lra \Hcl_*(X;\Z).$$
We then show that the homomorphism~$\Phi_*$ is injective.

\begin{lmm}\label{isom_l1}
Let $X$ be a manifold and $k\!\in\!\Z^{\ge0}$.
Suppose $c$ is an integer locally finite singular $k$-chain~$c$ as in~\e_ref{Zsingch_e}
with $\si_i\!\in\!C^{\i}(\De^k;X)$ for all $i\!\in\!\cI$ representing a cycle
in~$\ovScl_k(X;\Z)$ and $(M',M,F)$ is the triple corresponding to~$c$ via the construction 
of Lemma~\ref{constr1_lmm}.
The homology class~$[F|_M]$ obtained via the construction of Lemma~\ref{euler} then
satisfies
\BE{isoml1_e1}\big[F|_M\big]=[c]\in\Hcl_k(X;\Z).\EE
\end{lmm}

\begin{proof} For $k\!=\!0$, the claim clearly holds on the chain level.
Thus, suppose $k\!\ge\!1$.
Since the self-map $\vp_k$ of Lemma~\ref{bdpush_lmm} restricts to the identity on~$\prt\De^k$,
\BE{isom_l1e2}\vp_k\!-\!\id_k=\prt_{\De^k}s_k\in S_k(\De^k;\Z)\EE
for some $s_k\!\in\!S_{k+1}(\De^k;\Z)$.
Define
\begin{alignat*}{2}
\hb\!:\Hom(\De^k,X)&\lra S_k(\De^k;\Z), &\qquad\hb(\si)&=\vp_k, \\
\wt\hb\!:\Hom(\De^k,X)&\lra S_{k+1}(\De^k;\Z),&\qquad  \wt\hb(\si)&=s_k\,.
\end{alignat*}
By Lemma~\ref{hrigid_lmm} and~\e_ref{isom_l1e2}, the homomorphisms
$$\hb_{\#}\!:\Scl_k(X;\Z)\lra \Scl_k(X;\Z) \qquad\hbox{and}\qquad
\wt\hb_{\#}\!:\Scl_k(X;\Z)\lra \Scl_{k+1}(X;\Z)$$
induced via~\e_ref{hrigid_e3} and~\e_ref{hmap_e2b} are well-defined and satisfy
$$\hb_{\#}(c')-c'=\prt_X\big(\wt\hb_{\#}(c')\!\big)\in  \Scl_k(X;\Z)
\qquad\forall~c'\!\in\! \Scl_k(X;\Z).$$
In particular,
\BE{isom_l1e3}
\sum_{i\in\cI}\si_i\!\circ\!\vp_k-\sum_{i\in\cI}\si_i
\equiv \hb_{\#}(c)\!-\!c\in\prt \Scl_{k+1}(X;\Z).\EE

\vspace{.2in}

\noindent
Let $\pi$ be the quotient map of the proof of Lemma~\ref{constr1_lmm} and
$U\!\subset\!X$ be a neighborhood of $\Bd F|_M$ as in the proof of Lemma~\ref{euler}.
Choose a manifold with boundary $\ov{V}\!\subset\!M$ containing~\hbox{$M\!-\!F^{-1}(U)$} 
as in the latter proof so that $(\ov{V},\prt\ov{V})$ admits a triangulation $T\!\equiv\!(K,\eta)$ 
with each $k$-simplex of~$T$ contained in $\pi(\{i\}\!\times\!\De^k)$ for some $i\!\in\!\cI$.
Let
$$K^{\top}=\big\{\si\!: \dim\si\!=\!k\big\}.$$
For each $\si\!\in\!K^{\top}$, choose a linear map 
\BE{isom_l1e4}l_{\si}\!:\De^k\lra \si\subset|K|\EE
so that the map $\eta\!\circ\!l_{\si}\!:\De^k\!\lra\!M$ is orientation-preserving.
For each $i\!\in\!\cI$, let
$$K_i=\big\{\si\!\in\!K\!: \eta(\si)\!\subset\!\pi(\{i\}\!\times\!\De^k)\!\big\},
\qquad
K_i^{\top}=\big\{\si\!\in\!K_i\!: \dim\si\!=\!k\big\}.$$
Let $\wt{T}_i\!\equiv\!(\wt{K}_i,\wt\eta_i)$ be a triangulation of a subset of
$\{i\}\!\times\!\De^k$ that along with~$K_i$ gives a triangulation of~$\{i\}\!\times\!\De^k$. 
Put 
$$\ti{K}_i^{\top}=\big\{\si\!\in\!\ti{K}_i\!: \dim\si\!=\!k\big\}.$$
By definition of~$T$ and~$F$,
\BE{isom_l1e5}
\wt\eta_i(\si)\subset F^{-1}(U), \quad
\big\{\si_i\!\circ\!\vp_k\big\}\!\big(\wt\eta_i(\si)\!\big)
\subset U \qquad\forall~\si\!\in\!\wt{K}_i^{\top},~i\!\in\!\cI.\EE
Furthermore, by~\e_ref{isom_l1e3}
\BE{isom_l1e7}\begin{split}
\{c\}&=\sum_{i\in\cI}\big\{\si_i\!\circ\!\vp_k\big\}
=\sum_{i\in\cI}\sum_{\si\in K_i^{\top}}
\!\!\!\!\!\big\{\!\si_i\!\circ\!\vp_k\!\circ\!\eta\!\circ\! l_{\si}\!\big\}
+\sum_{i\in\cI}\sum_{\si\in\ti{K}_i^{\top}}
\!\!\!\!\!\big\{\!\si_i\!\circ\!\vp_k\!\circ\!\ti\eta_i\!\circ\!l_{\si}\!\big\}
\in\Hcl_k(X;\Z);
\end{split}\EE
the second equality above holds because subdivisions of cycles do not change the homology class.
By the proof of Lemma~\ref{euler}, the first sum on the right-hand side 
of~\e_ref{isom_l1e7} represents the image~$[F|_M]_{X;U}$
of~$[F|_M]$ under the isomorphism~\e_ref{euler_e2}.
By~\e_ref{isom_l1e5},  the second sum lies in $\ovScl_{\{U\};k}(X;\Z)$.
Since the sum of these two sums represents a cycle in $\ovScl_k(X)$,
it must represent~$[F|_M]$ in~$\ovScl_k(X;\Z)$. 
This gives~\e_ref{isoml1_e1}.
\end{proof}

\begin{lmm}\label{isom_l2}
Let $X$ be a manifold and $k\!\in\!\Z^{\ge0}$.
Suppose \hbox{$f\!:M\!\lra\!X$} is a Borel-Moore $k$-pseudocycle such that 
the homology class~$[f]$ provided by Lemma~\ref{euler} vanishes.
Then $f$ represents the zero element of~$\cHcl_k(X)$.
\end{lmm}

\begin{proof}
The case $k\!=\!0$ is straightforward and very similar to the $k\!=\!0$ case of 
the proof of Lemma~\ref{equiv1_lmm}.
Thus, we assume that $k\!\ge\!1$.
By Example~\ref{BMpseudo_eg}, we can also assume that $f^{-1}(\Bd f)\!=\!\eset$.\\

\noindent
By the first countability of the topology of~$X$ and Proposition~\ref{neighb2_prp},
there exists a sequence~$\{U_r\}_{r\in\Z^+}$ of open neighborhoods of~$\Bd f$ in~$X$
such that 
\BE{isoml2_e2}U_r\supset\ov{U}_{r+1}~~\forall\,r\!\in\!\Z^+, 
\qquad \bigcap_{r=1}^{\i}\!U_r=\Bd f,
\quad\hbox{and}\quad
\Hcl_{\{U_r\};l}(X;\Z)=0 \quad\forall~l>k\!-\!2.\EE
By the first condition above, the closed subset $M\!-\!f^{-1}(U_r)\!\subset\!M$ 
is contained in the open subset \hbox{$M\!-\!f^{-1}(\ov{U}_{r+1})$}.
Thus, we can choose submanifolds with boundary $\ov{V}_r\!\subset\!M$ as
in the proof of Lemma~\ref{euler} so~that 
$$M\!-\!f^{-1}(U_r)\subset V_r\subset \ov{V}_r\subset M\!-\!f^{-1}(\ov{U}_{r+1})
\quad \forall\,r\!\in\!\Z^+.$$
By the second condition in~\e_ref{isoml2_e2}, 
$$\bigcup_{r=1}^{\i}\!V_r\supset 
\bigcup_{r=1}^{\i}\!\big(M\!-\!f^{-1}(U_r)\!\big)
=M\!-\!f^{-1}(\Bd f)=M,$$
i.e.~the open collection $\{V_r\}_{r\in\Z^+}$ covers~$M$.\\

\noindent
Choose a triangulation $T\!=\!(K,\eta)$ of~$M$ that extends triangulations
of all~$\prt\ov{V}_r$ (which are pairwise disjoint).
Let
$$K^{\top}=\big\{\si\!\in\!K\!:\dim\si\!=\!k\big\}, \qquad
\cB_{\eta}=\big\{\!(\si,p)\!:\si\!\in\!K^{\top},\,p\!=\!0,1,\ldots,k\big\}.$$
For each $\si\!\in\!K^{\top}$, let $l_{\si}$ be as in~\e_ref{isom_l1e4}.
Put
\begin{gather}\label{isom2_e6}
f_{\si}=f\!\circ\!\eta\!\circ\!l_{\si}\!:\De^k\lra X
\quad\forall~\si\!\in\!K^{\top} \qquad\hbox{and}\\
\notag
\D_{\eta}=\big\{\!\big(\!(\si_1,p_1),(\si_2,p_2)\!\big)
\!\in\!\cB_{\eta}\!\times\!\cB_{\eta}\!:
(\si_1,p_1)\!\neq\!(\si_2,p_2),~
l_{{\si}_1}(\De^k_{p_1})\!=\!l_{{\si}_2}(\De^k_{p_2})\!\subset\!|K|\big\}.
\end{gather}
For each $(\!(\si_1,p_1),(\si_2,p_2)\!)\!\in\!\D_{\eta}$, define
$$\tau_{(\si_1,p_1),(\si_2,p_2)}\in\cS_{k-1} \qquad\hbox{by}\qquad
l_{\si_1}\!\circ\!\io_{k;p_1}\!\circ\!\tau_{(\si_1,p_1),(\si_2,p_2)}
=l_{\si_2}\!\circ\!\io_{k;p_2}.$$
Since $M$ is an oriented manifold and by the definition of $l_{\si_i}$,
$$\D_{\eta}\subset\cB_{\eta}\!\times\!\cB_{\eta} \qquad\hbox{and}\qquad
\tau\!:\D_{\eta}\lra{\cal S}_{k-1}$$
satisfy (i)-(iii) of Lemma~\ref{gluing_l1} with the subscript $c\!=\!\eta$ 
and the maps~$\si$ replaced by~$f_{\si}$.
Furthermore, the geometric realization~$|K|$ of~$K$ is the topological space~\e_ref{Mprdfn_e} 
with $(\cI,c)\!=\!(K^{\top},\eta)$ and
$$f\!\circ\!\eta\!\circ\!\pi|_{\si\times\De^k}=f_{\si} \qquad\forall~\si\!\in\!K^{\top},$$
where $\pi$ is the quotient map as in the proof of Lemma~\ref{constr1_lmm}.\\

\noindent
For each $r\!\in\!\Z^+$, let
$$K_r^{\top}=\big\{\si\!\in\!K^{\top}\!: \eta(\si)\!\subset\!\bar{V}_r\big\},
\quad \cB_{\eta;r}=\big\{\!(\si,p)\!\in\!\cB_{\eta}:\si\!\in\!K_r^{\top}\big\},
\quad
\D_{\eta;r}=\D_{\eta}\!\cap\!(\cB_{\eta;r}\!\times\!\cB_{\eta;r}).$$
By the construction of $[f]$ in the proof of Lemma~\ref{constr1_lmm}
and by the last paragraph of Section~\ref{OrientedBM_subs}, 
there exists a Borel-Moore singular chain 
$$c_r\equiv\sum_{i\in\cI_r}\!f_{r;i}\in \Scl_{\{U_r\};k}(X;\Z)$$ 
such that
\BE{crcond_e}\sum_{\si\in K_r^{\top}}\!\!\!\!\big\{f_{\si}\big\}
+\{c_r\}\in\ovScl_k(X;\Z)\EE
is a cycle representing~$[f]$. 
Similarly to Lemma~\ref{gluing_l1}, there exist a symmetric subset 
$$\D_r\subset(\cB_{\eta;r}\!\sqcup\!\cB_{c_r}) 
\!\times\! (\cB_{\eta;r}\!\sqcup\!\cB_{c_r})$$
disjoint from the diagonal and a map
$$\tau_r\!: \D_r\lra\cS_{k-1}$$ 
such that
\begin{enumerate}[label=(\roman*),leftmargin=*]

\item $\D_{\eta;r}\!\subset\!\D_r$ and  $\tau_r|_{\D_{\eta;r}}\!=\!\tau|_{\D_{\eta;r}}$;

\item the projection map $\D_r\!\lra\!\cB_{\eta;r}\!\sqcup\!\cB_{c_r}$
on either coordinate is a bijection;

\item for all $(\!(i_1,p_1),(i_2,p_2)\!)\!\in\!\D_r$, 
\begin{gather*}
\tau_{r;(i_1,p_1),(i_2,p_2)}^{~-1}=\tau_{r;(i_2,p_2),(i_1,p_1)}, \qquad
f_{r;i_1}\!\circ\!\io_{k;p_1}\!\circ\!\tau_{(i_1,p_1),(i_2,p_2)}
=f_{r;i_2}\!\circ\!\io_{k;p_2},\\
\hbox{and}\qquad \sign\tau_{r;(i_1,p_1),(i_2,p_2)}=-(-1)^{p_1+p_2},
\end{gather*}
where $f_{r;\si}\!\equiv\!f_{\si}$ for all $\si\!\in\!K_r^{\top}$.

\end{enumerate}

\vspace{.2in}

\noindent
Since every Borel-Moore singular chain~\e_ref{crcond_e} is a cycle,
$$\sum_{\si\in K_r^{\top}-K_{r-1}^{\top}}\hspace{-.3in}
\big\{f_{\si}\big\}+\!\{c_r\}-\{c_{r-1}\}  \in \ovScl_{\{U_{r-1}\};k}(X;\Z)$$
is a cycle as well. 
By the third condition in~\e_ref{isoml2_e2}, this cycle is a boundary.
Since $[f]\!=\!0$ by assumption, this conclusion also holds for $r\!=\!1$
with $U_0\!\equiv\!X$, $K_0^{\top}\!=\!\eset$, and $c_0\!=\!0$.
Let
\BE{isom2_e10}\ti{c}_r\equiv\sum_{i\in\wt\cI_r}\ti{f}_{r;i}\in \ovScl_{\{U_{r-1}\};k+1}(X;\Z)\EE
be a Borel-Moore singular chain such that 
$$\sum_{\si\in K_r^{\top}-K_{r-1}^{\top}}\hspace{-.3in}
\big\{f_{\si}\big\}+\!\{c_r\}-\{c_{r-1}\}
=\ov\prt_X\big\{\ti{c}_r\big\}\in \ovScl_{\{U_{r-1}\};k}(X;\Z).$$
Summing this equation with $r$ replaced by~$r'$ from~1 to~$r$, we obtain
\BE{isom2_e11}\sum_{\si\in K_r^{\top}}\!\!\!
\big\{f_{\si}\big\}+\!\{c_r\}
=\ov\prt_X\sum_{r'=1}^r\!\!\big\{\ti{c}_{r'}\big\}\in \ovScl_k(X;\Z)
\quad\forall\,r\!\in\!\Z^+.\EE

\vspace{.2in}

\noindent
Similarly to Lemma~\ref{gluing_l2}, \e_ref{isom2_e11} implies that 
there exist a subset
$$\wt\cB_r^f\subset
\wt\cB_r\!\equiv\!\bigsqcup_{r'=1}^r\!\!\cB_{\ti{c}_{r'}}\,,$$
a symmetric subset $\ti\D_r\!\subset\!\ti\cB_r\!\times\!\ti\cB_r$ disjoint from 
the diagonal, and maps
\begin{gather*}
\ti\tau_r\!: \ti\D_r\lra \cS_k, \quad 
\big(\!(i_1,p_1),(i_2,p_2)\!\big)\lra\ti\tau_{r;((i_1,p_1),(i_2,p_2))},\\
(\ti\io_r,\ti{p}_r)\!:K_r^{\top}\!\sqcup\!\cI_r\lra\ti\cB_r^f,
~~~\hbox{and}~~~
\ti\tau_r\!:K_r^{\top}\!\sqcup\!\cI_r\lra\cS_k, ~~i\lra\ti\tau_{(r,i)}, 
\end{gather*}
such that
\begin{enumerate}[label=(\roman*),leftmargin=*]

\item $\ti\D_{r-1}\!\subset\!\ti\D_r$, 
$\ti\tau_r|_{\ti\D_{r-1}}\!=\!\ti\tau_{r-1}$,
and $(\ti\io_r,\ti{p}_r,\ti\tau_r)|_{K_{r-1}^{\top}}=
(\ti\io_{r-1},\ti{p}_{r-1},\ti\tau_{r-1})|_{K_{r-1}^{\top}}$
if $r\!\ge\!2$;

\item the projection $\ti\D_r\!\lra\!\ti\cB_r$  on either coordinate
is a bijection onto the complement of~$\ti\cB_r^f$;

\item for all $(\!(i_1,p_1),(i_2,p_2)\!)\!\in\!\ti\D_r\!\cap\!
(\cB_{\ti{c}_{r_1}}\!\!\times\!\cB_{\ti{c}_{r_2}})$ with $r_1,r_2\!\in\![r]$,
\BE{isom2_e15}\begin{split}
&\ti\tau_{r;((i_1,p_1),(i_2,p_2))}^{~-1}=\ti\tau_{r;((i_2,p_2),(i_1,p_1))}, \quad
\ti{f}_{r_1;i_1}\!\circ\!\io_{k+1;p_1}\!\circ\!\ti\tau_{r;((i_1,p_1),(i_2,p_2))}
=\ti{f}_{r_2;i_2}\!\circ\!\io_{k+1;p_2},\\
&\hspace{1.2in}\hbox{and}\qquad \sign\ti\tau_{r;((i_1,p_1),(i_2,p_2))}=-(-1)^{p_1+p_2};
\end{split}\EE

\item for all $\si\!\in\!K_r^{\top}\!-\!K_{r-1}^{\top}$, 
\BE{isom2_e17}
\ti{f}_{r;\ti\io_r(\si)}\!\circ\!\io_{k+1;\ti{p}_r(\si)}\!\circ\!\ti\tau_{(r,\si)}=f_{\si}
\qquad\hbox{and}\qquad \sign\ti\tau_{(r,i)}=-(-1)^{\ti{p}_r(\si)};\EE

\item $(\ti\io_r,\ti{p}_r)$ is a bijection onto $\ti\cB_r^f$.

\end{enumerate}

\vspace{.2in}

\noindent
Put 
\begin{gather*}
\ti{M}'=\Big(\bigsqcup_{r=1}^{\i}\bigsqcup_{i\in\wt\cI_r}\!
\{r\}\!\times\!\{i\}\!\times\!\De^{k+1}\Big)\!\!\Big/\!\!\sim, 
\qquad\hbox{where}\\
\begin{split}
&\big(r_1,i_1,\io_{k+1;p_1}(\ti\tau_{r;((i_1,p_1),(i_2,p_2))}(t)\!)\!\big)
\sim \big(r_2,i_2,\io_{k+1;p_2}(t)\!\big)\\
&\hspace{1in}\forall~\big(\!(i_1,p_1),(i_2,p_2)\!\big)\!\in\!\ti\D_r\!\cap\!
(\cB_{\ti{c}_{r_1}}\!\!\times\!\cB_{\ti{c}_{r_2}}),\,
r_1,r_2,r\!\in\!\Z^+,\,t\!\in\!\De^k.
\end{split}\end{gather*}
Let $\ti\pi$ be the quotient map.
Define 
$$\ti{f}\!:\ti{M}'\lra X, \qquad
\ti{f}\big([r,i,t]\big)=\ti{f}_{r;i}\big(\vp_{k+1}(t)\!\big)
\quad\forall~t\!\in\!\De^{k+1},~i\!\in\!\wt\cI_r,~r\!\in\!\Z^+,$$
where $\vp_{k+1}$ is the self-map of $\De^{k+1}$ provided by Lemma~\ref{bdpush_lmm}.
Since~$\vp_{k+1}$ restricts to the identity on~$\prt\De^{k+1}$,
the map~$\ti{f}$ is well-defined by the second condition in~\e_ref{isom2_e15}
and continuous by the universal property of the quotient topology.
Similarly to the proof of Lemma~\ref{constr1_lmm}, the restriction of~$\wt{f}$ to
$$\wt\pi\bigg(\bigsqcup_{i\in\wt\cI_r}\!\{r\}\!\times\!\{i\}\!\times\!\De^{k+1}\bigg)
\subset\wt{M}'$$
is proper for every $r\!\in\!\Z^+$.
By~\e_ref{isom2_e10}, \hbox{$\wt{f}_{r';i}(\De^{k+1})\!\subset\!U_r$}  
for all $r'\!>\!r$.
Thus, 
\BE{isom2_e21}\Bd\wt{f}\subset  \bigcap_{r=1}^{\i}\!\ov{U}_r=\Bd f\,.\EE
 
\vspace{.2in}

\noindent
Let $\ti{M}\!\subset\!\wt{M}'$ be the complement of the subset
$$\ti\pi\Big(\bigsqcup_{r=0}^{\i}
\bigsqcup_{i\in\wt\cI_r}\!\{r\}\!\times\!\{i\}\!\times\!\ti{Y}\Big)
\subset\ti{M}',$$
where $\ti{Y}\!\subset\!\De^{k+1}$ is the $(k\!-\!1)$-skeleton as before.
By~Lemma~\ref{proper_lmm}\ref{Bdext_it} and~\e_ref{isom2_e21},
\BE{isom2_e25}\Bd\wt{f}|_{\wt{M}}\subset \big(\Bd\wt{f}\big)\cup
\bigcup_{r=1}^{\i}\bigcup_{i\in\wt\cI_r}\!\!\wt{f}_{r;i}\big(\vp_{k+1}(\ti{Y})\!\big)
\subset\big(\Bd f\big)\cup
\bigcup_{r=1}^{\i}\bigcup_{i\in\wt\cI_r}\!\!\wt{f}_{r;i}(\ti{Y}).\EE
Since $\wt{f}_{r;i}|_{\Int\De'}$ is smooth for all $i\!\in\!\wt\cI_r$, $r\!\in\!\Z^+$, 
and  all simplices $\De'\!\subset\!\De^{k+1}$,
$\Bd\ti{f}|_{\ti{M}}$ has dimension at most $k\!-\!1$ by~\e_ref{isom2_e25}.\\

\noindent
Let $Y_f\!\subset\!M$ denote the image
of the $(k\!-\!1)$-skeleton of~$|K|$ under~$\eta$.
The~map
\begin{gather*}
\io_f\!:M\!-\!Y_f\lra\wt{M}, \\
\io_f\big(\eta(l_{\si}(t)\!)\!\big)
=\big[r,\wt\io_r(\si),\io_{k+1;\wt{p}_r(\si)}\big(\wt\tau_{(r,\si)}(t)\!\big)\big]
\quad\forall~\si\!\in\!K_r^{\top}\!-\!K_{r-1}^{\top},\,r\!\in\!\Z^+,\,t\!\in\!\Int\De^k,
\end{gather*}
is a well-defined embedding.
By the first condition in~\e_ref{isom2_e17} and~\e_ref{isom2_e6},
$$\wt{f}\!\circ\!\io_f=f|_{M-Y_f}\,.$$
Thus, $\ti{f}|_{\ti{M}}$ is a Borel-Moore pseudocycle equivalence between
the Borel-Moore $k$-pseudocycles~$f$ and $\eset$, 
provided $\ti{M}$ is an oriented manifold,  $\ti{f}|_{\ti{M}}$ is a smooth~map,
$\io$ is a smooth embedding, and
$$\prt\wt{M}=\io_f(M\!-\!Y_f).$$
These are again straightforward local statements, 
which are established as in~(3) in the proof of \cite[Lemma~3.3]{Z}.
\end{proof}


\noindent
{\it Department of Mathematics, Stony Brook University, Stony Brook, NY 11790-3651\\
spencer.cattalani@stonybrook.edu, azinger@math.stonybrook.edu}


\begin{thebibliography}{99}

\bibitem{BM60} A.~Borel and J.~Moore, 
{\it Homology theory for locally compact spaces}, 
Michigan Math.~J.~7 (1960), no.~2, 137–-159

\bibitem{FZ1} M.~Farajzadeh Tehrani and A.~Zinger,
{\it On the rim tori refinement of relative Gromov-Witten invariants},
Commun.~Contemp.~Math.~23 (2021), no.~5, 2050051, 50pp

\bibitem{FZ2} M.~Farajzadeh Tehrani and A.~Zinger,
{\it On the refined symplectic sum formula for Gromov-Witten invariants},
 Internat.~J.~Math.~31 (2020), no.~4, 2050032, 60pp

\bibitem{Hatcher} A.~Hatcher,
{\it Algebraic Topology}, Cambridge University Press, 2002

\bibitem{Lee} J.~Lee, {\it Introduction to Smooth Manifolds},
GTM~218, 2nd Ed., Springer,~2013

\bibitem{McSa} D.~McDuff and D.~Salamon, 
{\it $J$-Holomorphic Curves and Quantum Cohomology},
University Lecture Series~6, AMS,~1994

\bibitem{MS12} D. McDuff and D.~Salamon, 
{\it J-holomorphic Curves and Symplectic Topology},
Colloquium Publications~52, AMS, 2012

\bibitem{MiSt} J.~Milnor and J.~Stasheff,  {\it Characteristic Classes}, 
Annals of Mathematics Studies, no.~76, 
Princeton University Press, 1974

\bibitem{Mu2} J.~Munkres, 
{\it Elements of Algebraic Topology}, Addison-Wesley, 1994

\bibitem{RT} Y.~Ruan and G.~Tian, 
{\it A mathematical theory of quantum cohomology}, 
J.~Diff.~Geom.~42 (1995), no.~2, 259--367

\bibitem{Spanier93} E.~Spanier, 
{\it  Singular homology and cohomology with local coefficients and duality for manifolds}, 
Pacific J. Math. 160 (1993), no.~1, 165–-200

\bibitem{Vicks} J.~Vicks, {\it Homology Theory},
GTM, Springer~1994

\bibitem{Wilkins} N.~Wilkins, 
{\it $S^1$-localization by pseudocycles, lifts to $S^1$-localization of moduli spaces,
and applications to invariants of $S^1$-equivariant symplectic cohomology},
math/2212.06044

\bibitem{Z} A.~Zinger,  
{\it Pseudocycles and integral homology}, 
Trans.~AMS 360 (2008), no.~5, 2741–-2765

\bibitem{Z2} A.~Zinger, 
{\it On transverse triangulations}, M\"unster J.~Math.~5 (2012), 99–-105


\end{thebibliography}
\end{document}